\documentclass[10pt,a4paper]{amsart}
\usepackage{amsmath,amssymb,dsfont,amsthm,graphicx,esvect,float,color,newtxmath,siunitx,subcaption,hyperref,underscore,commath,scalerel,yfonts,mathtools,braket,xcolor,nicefrac,listings,marvosym,anyfontsize,nicefrac,mathrsfs,bbm,tikz-cd,thm-restate}
\usepackage[utf8]{inputenc}
\usepackage[titletoc,title]{appendix}
\author{Manuel Cañizares}
\address{Manuel Cañizares\\Basque Center for Applied mathematics, Bilbao, Spain,}
\email{mcanizares@bcamath.org}
\date{\today}
\title[Local-Data Inverse Scattering with Rough Electric Potentials]{Local Near-Field Scattering Data Enables\\ Unique Reconstruction of Rough Electric Potentials}
\newtheorem{introthm}{Theorem}
\newtheorem{theorem}{Theorem}[section]

\newtheorem{lemma}[theorem]{Lemma}
\newtheorem{cor}[theorem]{Corollary}
\newtheorem{prop}[theorem]{Proposition}
\newtheorem*{thm:direct}{Theorem \ref{thm:1.1}}
\newtheorem*{thm:inverse}{Theorem \ref{thm:inverse}}

\theoremstyle{remark}
\newtheorem*{remark}{Remark}
\newtheorem*{acknowledgements}{Acknowledgements}

\newcommand{\be}{\begin{equation}}
\newcommand{\ee}{\end{equation}}
\newcommand{\supp}{\mathrm{supp}\,}
\newcommand*{\defeq}{\mathrel{\vcenter{\baselineskip0.5ex \lineskiplimit0pt
			\hbox{\scriptsize.}\hbox{\scriptsize.}}}%
	=}

\newcommand{\dd}{\,\mathrm{d}}

\newcommand\restr[2]{{
		\left.\kern-\nulldelimiterspace 
		#1 
		\right|_{#2} 
}}
\renewcommand{\norm}[1]{\lVert#1\rVert}

\newcommand{\ndev}{\partial_\nu}
\renewcommand{\abs}[1]{\left|#1\right|}

\newcommand\helm{(\Delta+\lambda-V)\,}
\newcommand\helmx{\left(\Delta_x+\lambda-V(x)\right)}
\newcommand\uto{u_{to}}

\newcommand{\Rd}{\mathbb{R}^d}
\newcommand{\sql}{\lambda^{1/2}}
\newcommand{\half}{{1/2}}
\newcommand{\hlf}[1]{#1/2}
\newcommand{\Ii}{\mathcal{I}}

\newcommand{\D}{\Omega}
\newcommand{\C}{\mathcal{C}}
\renewcommand{\S}{\mathbb{S}}
	
\begin{document}

\begin{abstract}
    The focus of this paper is the study of the inverse point-source scattering problem, specifically in relation to a certain class of electric potentials. Our research provides a novel uniqueness result for the inverse problem with local data, obtained from the near field pattern. Our work improves the work of Caro and Garcia, who investigated both the direct problem and the inverse problem with global near field data for \textit{critically singular} and $\delta$\textit{-shell} potentials.
  The primary contribution of our research is the introduction of a Runge approximation result for the near field data on the scattering problem which, in combination with an interior regularity argument, enables us to establish a uniqueness result for the inverse problem with local data. Additionaly, we manage to consider a slightly wider class of potentials.
\end{abstract}
	
\maketitle

    \section{Introduction}
  Along the paper, we will consider real potentials in dimension $d\geq3$ that can be written as \begin{equation}\label{V}V=V^0+\gamma^s+\alpha \dd\sigma,\end{equation} where $V^0\in L^{d/2} (\mathbb{R}^d;\mathbb{R}),$ $\dd\sigma$ denotes the surface measure of a compact hypersurface $\Gamma$ which is locally described by the graph of Lipschitz functions, $\alpha\in L^\infty(\Gamma;\mathbb{R})$ and $\gamma^s$ is of the form 
  \[ \gamma^s=\chi^2\, D^sg, \] for some $s<1$, $g\in L^\infty(\mathbb{R}^d;\mathbb{R})$, and $\chi\,\in \mathcal{C}^\infty(\Rd;[0,1])$ is a cut-off function. Here $D^s$ denotes the Riesz derivative, defined as $\widehat{D^sf}(\xi)=|\xi|^s\widehat{f}(\xi)$. The supports of $V^0$ and $\chi$ will be assumed to be compact.
  In essence, this means that the whole potential $V$ is compactly supported.
  Note that the novel potentials that we introduce in this work are just those of the class of $\gamma^s$, while $V^0$ and $\alpha\dd\sigma$ are those introduced by Caro and García in \cite{caroscattering2020}, which they call \textit{critically singular} and \textit{$\delta$-shell} potentials, respectively.\\
	The direct problem consists in finding the wave scattered by the potential, when an incident wave is emitted at fixed energy by a point source away from its support. Mathematically, this translates to finding $u_{sc}$ solving
	\begin{equation}\label{eq:helm}\tag{S}
	\begin{cases}
			\left(\Delta +\lambda- V\right)u_{sc}(\centerdot\,,y) = V u_{in}(\centerdot\,,y) & \textrm{in }\mathbb{R}^d,\\
			u_{sc}(\centerdot\,,y)\textrm{ satisfying SRC},
		\end{cases}
	\end{equation}
	Here, SRC stands for Sommerfeld Radiation Condition. A function $u$ is said to satisfy SRC if 
  \begin{equation}\label{eq:src}
    \lim_{|x|\to\infty}|x|^{\frac{d-1}{2}}\left(\frac{x}{|x|}\cdot\nabla u(x)-i\lambda^{1/2}u(x)\right)=0
  \end{equation}
   uniformly in every direction, and $u_{in}(x,y)=\Phi_\lambda(x-y)$ denotes the incident wave emitted from the point $y\in \mathbb{R}^d\setminus \supp V$, where $\Phi_\lambda$ represents the fundamental solution for the Helmholtz equation with SRC, which solves the distributional problem:
	\begin{equation*}\label{eq:homo}\tag{F}
		\begin{cases}
				\left(\Delta +\lambda\right)\Phi_\lambda = \delta_0 & \textrm{in }\mathbb{R}^d,\\
				\Phi_\lambda\textrm{ satisfying SRC}.
			\end{cases}
	\end{equation*}
	 The SRC is classically introduced so that the solution to the Helmholtz equation is unique and physically corresponds to a radiating wave \cite{schoteighty1992}.\par In section \ref{sect:direct}, we will arrive to the following theorem:
	\begin{introthm}\label{thm:1.1}
		Suppose $V$ is of the form \eqref{V}. Then, there exists $\lambda_0=\lambda_0(V,d)$ such that, for every $\lambda\geq \lambda_0$, there is an unique solution $u_{sc}(\centerdot\,,y)\in X_\lambda^*$ to the problem \eqref{eq:helm} for every $y\in \mathbb{R}^d\setminus \supp V$.\\
    Moreover, the mapping $Vu_{in}(\centerdot\,,y)\mapsto u_{sc}(\centerdot\,,y)$ is bounded from $X_\lambda$ to $X_\lambda^*$.
	\end{introthm} 
	The spaces $X_\lambda$ and $X_\lambda^*$ are introduced in \cite{caroscattering2020} and their definition will be given at the beginning of section 2. The elements in the space $X_\lambda^*$  can be thought of as functions with one derivative in $L^2$ that exhibit a certain degree of integrability in frequences $|\xi|\sim\lambda^{1/2}$, and $X_\lambda$ will be the pre-dual of $X_\lambda^*$. To prove the theorem, we will follow the argument in \cite{caroscattering2020} and make use of many of their results. In Caro and Garcia's work, the first step was to deal with the critically singular part of the potential by obtaining an inverse for the operator $\Delta+\lambda-V^0$ via a Neumann-series approach.
	\par If we denote the inverse above by \mbox{$(\Delta+\lambda+i0-V^0)^{-1}$}, it can be applied to both sides of \eqref{eq:helm} to turn the problem into finding the inverse of the operator \[ I-(\Delta+\lambda+i0-V^0)^{-1}\circ(\alpha\dd\sigma+\gamma^s). \] The appropiate framework to study this operator is the Fredholm theory, due to the fact that the operator defined as multiplication by $(\alpha\dd\sigma+\gamma^s)$ is compact from $X_\lambda^*$ to $X_\lambda$. Fredholm alternative assures the existence of such an inverse as long as the operator $I-(\Delta+\lambda+i0-V^0)^{-1}\circ(\alpha\dd\sigma+\gamma^s)$ is injective. Proving this injectivity will be equivalent to obtaining uniqueness for the problem 
	\begin{equation*}\label{eq:hpot}\tag{H}
		\begin{cases}
				\left(\Delta +\lambda-V\right)u = 0 & \textrm{in }\mathbb{R}^d,\\
				u\textrm{ satisfying SRC}.
			\end{cases}
	\end{equation*}
	The fundamental ingredient here is an unique continuation argument for the operator \mbox{$\Delta+\lambda-V$}. In our case, it will be based on a Carleman estimate proved by Caro and Rogers in \cite{caroglobal2016} for a modified family of Bourgain-type spaces. They were themselves based on those introduced by Haberman and Tataru in \cite{habermanuniqueness2013} to study the Calderón problem. As we will note below, this Carleman estimate will prove to be crucial for the solution of the inverse problem. \par 
  Caro and Garcia applied this framework to invert the operator when when $\alpha$ is in $L^\infty(\Gamma)$. We noticed that there was a possibility to consider a wider range of potentials here. The idea is that $\alpha\dd\sigma$ acts as multiplication operator from $H^s(\D)$ to $H^{-s}_0(\D)$ with $s<1$, where $\D$ is any domain containing the support of $V$. Then, the fact that $X_\lambda^*$ is related to $H^1$, along with $\alpha\dd\sigma$ being compactly supported, and the fact $H^1(\D)$ is compactly embedded in $H^s(\D)$, give the necessary ideas to prove compactness of the multiplication operator. In principle, one could hope to consider $\alpha\in L^r(\Gamma)$, with $r>d-1$, since $\alpha\dd\sigma$ would still behave as a compact multiplication operator form $H^1(\D)$ to $H^{-1}_0(\D)$. Nonetheless, the Carleman estimate doesn't seem to work in that case, and therefore we did not manage to relax this condition in the solution of the inverse problem.\par
   On the other hand, we attempted to consider distributions of the form $D^sg$, with $s<1$ and $g\in L^p(\Rd)$, hoping to find a class of such distributions that generalized those of the form $\alpha\dd\sigma$. However, in principle, the set of indexes $(s,p)$ that made the Carleman estimate work for these distributions did not allow them to \textit{see} the hypersurface $\Gamma$. Therefore, we opted to consider potentials that could be decomposed as a sum of those of both classes.\\
    The arguments to find the scattering solution via Fredholm theory can be found in section \ref{sect:direct}. Here, the structure is akin to that in \cite{caroscattering2020}, but some proofs have to be redone in order for them to work for our wider class of potentials.\\
  Afterwards, in sections \ref{sect:ortho} and \ref{sect:cgo}, we will devote ourselves to proving the following theorem concerning uniqueness with partial data for the inverse problem. 
	\begin{introthm}\label{thm:1.2}
		Consider $d\geq 3$. Let $\Sigma_1,\Sigma_2$ be two relatively open sets of dimension $d-1$, separated from $\supp V$, and that can be expressed as the graph of $\C^2$ functions, and let $V_1$ and $V_2$ be two potentials of the form described in \eqref{V}. Let also $u_{sc,1}$ and $u_{sc_2}$ be the scattering solutions to the problem \eqref{eq:helm} with respective potentials $V_1$ and $V_2$.
		Then, there exists $\lambda_0=\lambda_0(V,d)$ such that, for all $\lambda\geq\lambda_0$ except for at most a countable set, it holds that
		\[ \restr{u_{sc,1}}{\Sigma_1\times\Sigma_2}=\restr{u_{sc,2}}{\Sigma_1\times\Sigma_2} \implies V_1 = V_2. \]
	\end{introthm}
	From a physics perspective, we are stating that the identifiability of the potential is possible by placing sources and detectors only in small pieces of hypersurfaces away from the support of $V$. These sources will be emitting monochromatic waves with a fixed energy $\lambda$. Note that $\Sigma_1$ and $\Sigma_2$ could well be the same set, they could be different but intersecting, or be completely separate and non-intersecting.\\
	We will prove this theorem via an orthogonality relation in the spirit of Alessandrini's identity for the Calderón problem \cite{alessandrinistable1988} and the construction of CGO solutions as in \cite{sylvesterglobal1987,caroglobal2016,hahnernew2001}. This orthogonality relation will be proven in section \ref{sect:ortho}. The main ingredient will be a Runge approximation result, indexed as proposition \ref{prop:runge}, given in section \ref{sect:runge}. We will take a bounded open domain $\D$, of class $\mathcal{C}^2$, whose boundary contains both $\Sigma_1$ and $\Sigma_2$, and the Runge approximation will allow us to approximate solutions in $\D$ by single layer potentials with densities that are supported in any subset $\Sigma\subset\partial \D$:
	\[ \mathcal{S}f(x)=\int_{\partial \D}f(y)\,u_{to}(x,y)\dd y, \]
	where $f\in \mathcal{C}(\partial \D)$, $\supp f\subset\Sigma$. Here the total wave is defined as $u_{to}\defeq u_{in}+u_{sc}$.
	The result will give an approximation in the $L^2$ norm of a smaller open domain $\D'$ strictly contained in $\D$ and containing the support of $V$. The proof of this lemma was inspired by that of Isakov for the Calderón problem for $C^2$ conductivities \cite{isakovuniqueness1988} and that of Harrach, Pohjola and Salo for the recovery of an $L^\infty$ scattering coefficient function in the Helmholtz equation \cite{harrachmonotonicity2019}. Nonetheless, the argument is slightly different, since we consider approximating solutions of a different kind. Also, there are some technicalities concerning the singularities of the fundamental solutions, which have to be treated with a little bit of care. This argument will allow us to extend an orthogonality relation of the type
  \begin{equation}\label{eq:orthog} \braket{(V_1-V_2)v_1,v_2}=0, \end{equation} with $v_1$,$v_2$ being solutions that only see $\Sigma_j$, to an orthogonality relation of this type for all solutions to the equation.\par
	However, the approximation in $L^2$ won't be enough, due to the low integrability of our potentials. We will in fact need to approximate our solutions in an $H^1$ norm so that the solutions can be integrated against the potential. Therefore, we will provide in lemma \ref{lemma:int} with an interior regularity result. In particular, for any $u$ solution of \eqref{eq:helm}, we will obtain \[ \norm{u}_{H^1(\Omega)}\lesssim\norm{u}_{L^2(\D')}, \] as long as $\Omega$ is a domain strictly contained in $\D'$ and also containing the support of $V$. We will prove this inequality following an argument by Chen that appears in \cite{chensecond1998}, who proved $H^2$ regularity for second order elliptic equations with a degree of regularity that in our case would ask for $V\in L^\infty$. We adapt the argument to work for our kind of potentials, which can be done thanks to the fact that $\alpha\dd\sigma+\gamma^s$ acts as a bilinear form over $H^s$, which is an interpolation space between $L^2$ and $H^1$. \\
  Then, the obtention of the final orthogonality relation will be given in section \ref{sect:proportho}. 
  \begin{remark}
  It might be interesting to note that, in our case, the only sets that are given by the problem and that have a clear physical meaning are the supports of the potentials and the measuring sets $\Sigma_1$ and $\Sigma_2$. We will construct \textit{ad hoc} the domains in which we obtain the orthogonality relation \eqref{eq:orthog}. To make the arguments work, we need the frequency $\lambda$ not to be a Neumann eigenvalue for these domains. Since the set of Neuman eigenvalues in each domain is countable, we can already assert that the proofs go through for all $\lambda\geq\lambda_0$ \textit{except for at most a countable set}, as in the statement of Theorem \ref{thm:1.2}. In \cite{Stefanov1990}, Stefanov used the monotonicity of Dirichlet eigenvalues with respect to domain inclusion to prove that, for any value of $\lambda$, it is possible to construct a domain such that $\lambda$ is not a Dirichlet eigenvalue. However, this monotonicity doesn't seem to hold for Neumann eigenvalues for the Laplacian \cite{freitas2023domain}, and we have no reason to think that it holds for $-\Delta+V$. Because of this, we have not been able to prove the existence of an appropiate domain for every value of $\lambda$. Nonetheless, given the amount of freedom that one has when choosing the domain, it would be reasonable to expect that such a domain exists. The analyis of this question is left for a future work.
  \end{remark}
	Finally, to end the proof of Theorem \ref{thm:1.2}, we will test the aforementioned identity with a special type of solutions, the so-called complex geometric optics (CGO) solutions. This is a classical method that goes back to Sylvester and Uhlmann's work \cite{sylvesterglobal1987} in the Calderón problem. In our case, these special solutions are of the form \begin{equation*}v_j(x)=e^{\zeta_j\cdot x}(1+w_j(x)),\end{equation*} where $\zeta_j\in\mathbb{C}^d$ are chosen such that $\zeta_j\cdot\zeta_j=-\lambda$ and $\zeta_1+\zeta_2=-i\kappa$ for an arbitray $\kappa\in\Rd$ -which is possible in dimension $d\geq3$- and the correction term $w_j$ vanishes in a certain sense when $|\zeta_j|$ grows. To prove the existence of these solutions, we will follow the construction of both Caro and García in \cite{caroscattering2020} and Caro and Rogers in \cite{caroglobal2016}. The key ingredient again will be the aforementioned Carleman estimate. In particular, applying the operator $(\Delta+\lambda-V)$ to $v_j$ yields
	\[ (\Delta+2\zeta_j\cdot\nabla-V)w_j=V. \]
	Therefore, to construct the CGO solutions it is enough to prove injectivity of the adjoint operator $(\Delta-2\zeta_j\cdot\nabla-V)$, which can be done via a priori estimates. We will define a family of Bourgain spaces $X_\zeta^s$ via the norm
	\[ \norm{u}_{X_\zeta^s}=\norm{(M|\Re(\zeta)|^2+M^{-1}|p_\zeta|^2)^{s/2}\widehat{u}}_{L^2} \]
	with $M>1$, where $\Re$ denotes the real part and 
	\[ p_\zeta(\xi)=-|\xi|^2+2i\zeta\cdot\xi+\zeta\cdot\zeta. \]
	The index $s=\half$ will play an important role. Realising that the dual of $X_\zeta^s$ is $X_{-\zeta}^{-s}$, if we want to prove the existence of solutions in $X_\zeta^\half$ with potentials $V\in X_\zeta^{-\half}$, we will need to prove the injectivity of the operator from $X_{-\zeta}^{-\half}$ to $X_{-\zeta}^\half$. Therefore, the a priori estimate that we will seek will be of the form
	\[ \norm{u}_{X_{-\zeta}^\half}\lesssim\norm{(\Delta-2\zeta\cdot\nabla-V)\,u}_{X_{-\zeta}^{-\half}} \]
  This will be done in section \ref{sect:cgo}.
  \begin{acknowledgements}
    The author would like to thank his supervisor Pedro Caro for his guidance and invaluable ideas. He would also like to thank Antti Kykkänen and María Ángeles García Ferrero for their comments. This research is supported by MCIN/AEI under FPI fellowship PRE2019-091776 and the project PID2021-122156NB-I00 and also by the Basque Government through the BERC 2022-2025 program and by the Ministry of Science and Innovation: BCAM Severo Ochoa accreditation CEX2021-001142-S a/ MICIN / AEI / 10.13039/501100011033.
  \end{acknowledgements}

    \section{Direct Scattering}\label{sect:direct}
    In this section, we follow the strategy by Caro and García to solve the direct problem for the point-source scattering. As the name of the problem suggests, we consider an incident wave emitted by a point source at fixed energy $\lambda>0$. This means that the spatial part of the incident wave emitted at a point $y\in\Rd$ and measured at another point $x\in\Rd$ will be given by $u_{in}(x,y)=\Phi_\lambda(x-y)$, where $\Phi_\lambda$ is the radiating fundamental equation to the Helmholtz equation, i.e. $\Phi_\lambda$ is the solution to the problem
    \begin{equation}\label{SRC}
        \begin{cases}
                \left(\Delta +\lambda\right)\Phi_\lambda = \delta_0 & \textrm{in }\mathbb{R}^d,\\
                \Phi_\lambda\textrm{ satisfying SRC}.
            \end{cases}
    \end{equation}
    We denote by $\left(\Delta+\lambda+i0\right)^{-1}$ the solution operator for the Helmholtz equation with SRC \eqref{eq:src}. Under this notation, we can write $\Phi_\lambda=\left(\Delta+\lambda+i0\right)^{-1}\delta_0$, which for instance can be expressed in terms of the Fourier symbol for the operator $\Delta+\lambda$. If we denote the modulus of this symbol by $m_\lambda(\xi)=|\lambda-|\xi|^2|$, it can be checked that the fundamental solution $\Phi_\lambda$ is given, in the distributional sense, as
    \begin{equation} \label{eq:fundsol} 
      \braket{\Phi_\lambda,f}=\frac{1}{(2\pi)^{d/2}}\left[\lim_{\varepsilon\to0}\int_{m_\lambda>\varepsilon}\frac{\widehat f(\xi)}{\lambda-|\xi|^2}\dd\xi-i\frac{\pi}{2\lambda^{1/2}}\int_{S_\lambda}\widehat f(\xi)\dd S_\lambda(\xi)\right],
    \end{equation}
    for $f\in\mathscr{S}(\Rd)$. Here, $S_\lambda=\{\xi\in\Rd\,:\,|\xi|=\lambda^{1/2}\}$ is the critical hypersurface of the symbol $m_\lambda$, and $\dd S_\lambda$ denotes its volume form.
    \subsection{The free resolvent and Neumann series} As mentioned in the introduction, the scattering solution is constructed in certain spaces $X_\lambda^*$ that were considered by Caro and García in \cite{caroscattering2020}. In those spaces, refinements of estimates by Agmon and Hörmander \cite{agmonasymptotic1976}, by Kenig, Ruiz and Sogge \cite{keniguniform1987} and by Ruiz and Vega \cite{ruizlocalnodate} give us a good estimate for the free resolvent $(\Delta+\lambda+i0)^{-1}$. We are going to recall the definition of these spaces. For this, we need to construct a partition of unity.
    Indeed, choose a function $\phi\in\mathscr{S}(\Rd)$ supported in $\{\xi\in\Rd\,:\,|\xi|\leq2\}$ such that $\phi(\xi)=1$ whenever $|\xi|\leq1$ and define, for $k\in\mathbb{Z}$, $\psi_k(\xi)\defeq\phi(2^{-k}\xi)-\phi(2^{-k+1}\xi)$. Note that $\psi_k$ is supported in ${\{\xi\in\Rd\,:\,2^{k-1}\leq|\xi|\leq2^{k+1}\}}$ and ${\sum_{k\in\mathbb{Z}}\psi_k(\xi)=1}$ for $\xi\neq0$. We define the Littlewood-Paley projectors as
    \begin{align}\label{eq:LP}
      \begin{split}
      \widehat{P_kf}(\xi)&\defeq \psi_k(\xi)\widehat f(\xi),\\
      \widehat{P_{\leq k}f}(\xi)&\defeq \sum_{j\leq k}\widehat{P_jf}(\xi)=\phi(2^{-k}\xi)\widehat f(\xi).
      \end{split}
    \end{align}
    Now, for $\lambda>0$, let $k_\lambda\in\mathbb{Z}$ be such that $2^{k_\lambda-1}<\sql\leq2^{k_\lambda}$. Then, the set of indices in the projectors above that \textit{see} the critical frequencies is \[ I=\{k_\lambda-2,k_\lambda-1,k_\lambda,k_\lambda+1\}. \]
    For simplicity, we will call $P_{<I}\defeq P_{\leq k_\lambda-3}$. Define now the space $B$ and its dual $B^*$, as in \cite{agmonasymptotic1976}, via the norms
    \[ \norm{f}_{B}=\sum_{j\in\mathbb{N}_0}\,\left(2^{j/2}\norm{f}_{L^2(D_j)}\right),\qquad \norm{u}_{B^*}=\sup_{j\in\mathbb{N}_0}\,\left(2^{-j/2}\norm{u}_{L^2(D_j)}\right), \]
    with $D_j=\{x\in\Rd\,:\,2^{j-1}<|x|\leq 2^j\}$ for $j\in \mathbb{N}$ and $D_0=\{x\in\Rd\,:\,|x|\leq 1\}.$ From now on, denote by $q_d$ the end-point index for the Stein-Tomas trace theorem, $2/q_d=(d-1)/(d+1)$, by $p_d$ the end-point index for the $\dot{H}^1$-Hardy-Littlewood-Sobolev embedding theorem, $1/p_d=1/2-1/d$ for $d\geq3$, and by $q'_d$ and $p'_d$ their respective Hölder conjugates. The space $X_\lambda$ can be defined as the sum of two spaces, $Y_\lambda$ and $Z_\lambda$, which are defined as elements in $f\in\mathscr{S}'(\Rd)$ with norms
    \[ \norm{f}^2_{Y_\lambda}\defeq \norm{m_\lambda^{-\half}\widehat{P_{<I}f}}^2_{L^2}+\sum_{k\in I}\lambda^{-\half}\norm{P_kf}^2_B+\sum_{k>k_\lambda+1}\norm{m_\lambda^{-\half}\widehat{P_kf}}^2_{L^2}, \] and
    \[ \norm{f}^2_{Z_\lambda}\defeq \norm{m_\lambda^{-\half}\widehat{P_{<I}f}}^2_{L^2}+\sum_{k\in I}\lambda^{d(\frac{1}{q_d'}-\frac{1}{p_d'})}\norm{P_kf}^2_{L^{q_d'}}+\sum_{k>k_\lambda+1}\norm{m_\lambda^{-\half}\widehat{P_kf}}^2_{L^2}. \]
    The norm in $X_\lambda$ will be the usual for the sum of normed spaces:
    \[ \norm{f}_{X_\lambda}=\inf_{g+h=f}\{\norm{g}_{Y_\lambda}+\norm{h}_{Z_\lambda}\}. \]
    The spaces above have respective dual spaces $Y_\lambda^*$ and $Z_\lambda^*$ defined by the norms
    \[ \norm{u}^2_{Y^*_\lambda}\defeq \norm{m_\lambda^{\half}\widehat{P_{<I}f}}^2_{L^2}+\sum_{k\in I}\lambda^{\half}\norm{P_kf}^2_{B^*}+\sum_{k>k_\lambda+1}\norm{m_\lambda^{\half}\widehat{P_kf}}^2_{L^2},\] and
    \[ \norm{u}^2_{Z^*_\lambda}\defeq \norm{m_\lambda^{\half}\widehat{P_{<I}f}}^2_{L^2}+\sum_{k\in I}\lambda^{d(\frac{1}{q_d}-\frac{1}{p_d})}\norm{P_kf}^2_{L^{q_d}}+\sum_{k>k_\lambda+1}\norm{m_\lambda^{\half}\widehat{P_kf}}^2_{L^2}. \]
    Then, the space $X^*_\lambda$ can be defined as the dual space of $X_\lambda$, which will be isomorphic to the intersection of $Y_\lambda^*$ and $Z_\lambda^*$, with norm
    \begin{align*} 
      \norm{u}^2_{X^*_\lambda}=&  \norm{m_\lambda^{\half}\widehat{P_{<I}f}}^2_{L^2}+ \sum_{k\in I}\left(\lambda^{\half}\norm{P_kf}^2_{B^*} + \lambda^{d(\frac{1}{q_d}-\frac{1}{p_d})}\norm{P_kf}^2_{L^{q_d}}\right)+ \\ & \sum_{k>k_\lambda+1}\norm{m_\lambda^{\half}\widehat{P_kf}}^2_{L^2} \sim \norm{u}^2_{Y^*_\lambda}+\norm{u}^2_{Z^*_\lambda}.
    \end{align*}
    It will be interesting to note that the Schwartz class is dense in all the above spaces with respect to their corresponding norms \cite{caroscattering2020}. We will now obtain an estimate for the solution operator $X_\lambda$ and $X_\lambda^*$. However, let us first prove the following lemma.
    \begin{lemma}\label{lem:bernstein}
      Let $1\leq p\leq \infty$ and $s\in \mathbb{R}$. Define $D^s$ as the Fourier multiplier with symbol $|\xi|^s$. Then, it holds that
      \[ \norm{D^sP_k f}_{L^p}\sim 2^{ks}\norm{P_k f}_{L^p}. \]
    \end{lemma}
    \begin{proof}
      Fix $s\in \mathbb{R}$. We have, by definition,
      \[ \widehat{D^sP_kf}(\xi)=|\xi|^s \psi_k(\xi)\widehat f(\xi)=2^{ks}\left(\frac{|\xi|}{2^k}\right)^s\psi\left(\frac{\xi}{2^k}\right)\widehat f(\xi). \]
      If we take $\chi$ to be a smooth cut-off function such that $0\leq\chi\leq1$, $\chi=0$ around the origin and $\chi=1$ on $\supp \psi$, then
      \[ \widehat{D^sP_kf}(\xi)=2^{ks}\left(\frac{|\xi|}{2^k}\right)^s\chi\left(\frac{\xi}{2^k}\right)\psi\left(\frac{\xi}{2^k}\right)\widehat f(\xi). \]
      Note that the function $\rho(\xi)\defeq|\xi|^s\chi(\xi)$ is compactly supported away from $0$ and therefore $\rho\in \mathcal{C}^\infty_c(\Rd)$ for any value of $s$. Now,
      \[ \widehat{D^sP_kf}(\xi)=2^{ks}\rho\left(\frac{\xi}{2^k}\right)\widehat {P_kf} (\xi), \] and thus \[ D^sP_k f \sim 2^{ks}\left[(P_k f)\ast 2^{kd}\check \rho (2^k\centerdot)\right]. \]
      Therefore, by Young's inequality,
      \[ \norm{D^sP_kf}_{L^p}\lesssim 2^{ks}\norm{P_kf}_{L^p}\norm{2^{kd}\check \rho (2^k\centerdot)}_{L^1}=2^{ks}\norm{P_kf}_{L^p}\norm{\check \rho}_{L^1}\lesssim2^{ks}\norm{P_kf}_{L^p}.\]
      Now, to get the reverse inequality, observe that
      \[ \widehat{P_k f}(\xi) = |\xi|^{-s}|\xi|^s \widehat{P_kf}(\xi)=\mathcal{F}[D^{-s}D^sP_k f] (\xi), \]
      where $\mathcal{F}$ denotes the Fourier transform. Therefore,
      \[ \norm{P_k f}_{L^p} \sim \norm{D^{-s}D^sP_k f}_{L^p}\lesssim 2^{-ks} \norm{D^sP_k f}_{L^p}, \]
      and thus 
      \[ 2^{ks} \norm{P_k f}_{L^p} \lesssim  \norm{D^sP_k f}_{L^p}. \]
    \end{proof}
    We are ready now for the desired estimate.
    \begin{theorem}\label{thm:ineqX}
      There exists a constant $C>0$ depending only on $d$ such that \[ \norm{(\Delta+\lambda+i0)^{-1}f}_{X^*_\lambda}\leq C \norm{f}_{X_\lambda} \] for every $f\in\mathscr{S}(\Rd)$.
    \end{theorem}
    \begin{proof}
      In \cite{caroscattering2020}, Caro and García proved the inequalities 
      \[ \norm{(\Delta+\lambda+i0)^{-1}f}_{Y^*_\lambda}\lesssim_d \norm{f}_{Y_\lambda}, \]
      and 
      \[ \norm{(\Delta+\lambda+i0)^{-1}f}_{Z^*_\lambda}\lesssim_d \norm{f}_{Z_\lambda}. \]
      Now, for the diagonal inequalities, note that $P_k(\Delta+\lambda+i0)^{-1}f=(\Delta+\lambda+i0)^{-1}P_kf$ and therefore it follows that, for any $k\notin I$, 
      \[ \braket{P_k (\Delta+\lambda+i0)^{-1}f,\overline g}=\frac{1}{(2\pi)^{d/2}}\int_{\Rd}\frac{\widehat {P_kf}(\xi)\,\overline{\widehat g(\xi)}}{\lambda-|\xi|^2}\dd\xi, \]
      since the frequencies of $P_k f$ are separated from $S_\lambda$. This implies, by Plancherel's identity, that 
      \begin{equation}\label{eq:l2mult} \norm{m_\lambda^{\half}\mathcal{F}[P_k(\Delta+\lambda+i0)^{-1}f]}_{L^2} = \norm{m_\lambda^{-\half}\widehat{P_kf}}_{L^2}, \end{equation}
      On the other hand, Theorem 3.1 in \cite{ruizlocalnodate} states that, for $u\in\mathscr{S}(\Rd)$ a solution of $(\Delta+\lambda)\,u=f,$
      \[ \norm{D^{\half}u}_{B^*}\lesssim_d \lambda^{\frac{d}{2}(\frac{1}{q_d}-\frac{1}{p_d})}\norm{f}_{L^{q_d'}}, \]
      where $D^{\half}$ is the Fourier multiplier with symbol $|\xi|^{\half}$. By duality, we have that
      \[ \norm{D^{\half}u}_{L^{q_d}}\lesssim_d \lambda^{\frac{d}{2}(\frac{1}{q_d}-\frac{1}{p_d})}\norm{f}_{B}. \]
      Note again that, if $u$ solves $(\Delta+\lambda)\,u=f,$ then $(\Delta+\lambda)P_ku=P_kf,$ for any $k\in\mathbb{Z}$. Thus, for $k\in I$, we have that, since $2^k\sim \lambda^{1/2}$,
      \[  \lambda^{1/4} \norm{P_k u}_{L^{q_d}}\sim_d \norm{D^\half P_k u}_{L^{q_d}} \lesssim_d \lambda^{\frac{d}{2}(\frac{1}{q_d}-\frac{1}{p_d})}\norm{f}_{B}, \]
      where we have used lemma \ref{lem:bernstein}.
      This, along with \eqref{eq:l2mult} gives us the estimate
      \[ \norm{(\Delta+\lambda+i0)^{-1}f}_{Z^*_\lambda}\lesssim_d \norm{f}_{Y_\lambda}, \]
      and, by duality,
      \[ \norm{(\Delta+\lambda+i0)^{-1}f}_{Y^*_\lambda}\lesssim_d \norm{f}_{Z_\lambda}. \]
      Applying the definition of the spaces $X_\lambda$ and $X_\lambda^*$ gives us the desired result.
    \end{proof}
    Next, the compact support of the potential $V^0$ lets us split it into an $L^\infty$ component and a $L^{p_d}$ component whose norm can be as small as needed, which allows for an estimate of the type \[ \norm{V^0}_{\mathcal{L}(X^*_\lambda,X_\lambda)}\leq C(\lambda^{1/4}+\norm{\mathbbm{1}_FV^0}_{L^{d/2}}), \] where $F=\{x\in\Rd\,:\,|V^0(x)|>\lambda^{1/4}\}$.
    With this estimate, one can construct the solution operator $(\Delta+\lambda+i0-V^0)^{-1}$ via Neumann series and prove its boundedness from $X_\lambda$ to $X_\lambda^*$. All this was done by Caro and Garcia in \cite{caroscattering2020}, and can be summarized in the following proposition. Note that their work was done on a ball, but their argument would be identical for any Lipschitz domain. Therefore, for the rest of the section, we will consider a bounded domain $\D$ such that $\supp V\subset \D$.
    \begin{prop}[\cite{caroscattering2020}]\label{prop:neumann}
       Let $\D$ be a bounded Lipschitz domain. The operator defined by
      \[ (\Delta+\lambda+i0-V^0)^{-1}f=\sum_{n\in\mathbb{N}}[((\Delta+\lambda+i0)^{-1}\circ V^0)]^{n-1}\,((\Delta+\lambda+i0)^{-1}\,f) \] for any $f\in X_\lambda$ is bounded from $X_\lambda$ to $X_\lambda^*$. Moreover, $u=(\Delta+\lambda+i0-V^0)^{-1}f$ solves the equation \[ (\Delta+\lambda-V^0)\,u=f\textnormal{ in } \Rd \] and, if $f$ is compactly supported in $\D$, then $u$ satisfies SRC \eqref{eq:src}.
    \end{prop}
    \subsection{The Fredholm alternative}\label{sect:fred}
    Now we will construct the scattering solution $u_{sc}(\centerdot\,,y)$ as the solution of the equation
    \begin{equation}\label{eq:fred}
      [I-(\Delta+\lambda+i0-V^0)^{-1}\circ(\gamma^s+\alpha\dd\sigma)]\,u_{sc}(\centerdot\,,y)=f(\centerdot\,,y)\textnormal{ in }\Rd
    \end{equation}
    with $f(\centerdot\,,y)=(\Delta+\lambda+i0-V^0)^{-1}\,(V u_{in}(\centerdot\,,y)).$ Note that applying the operator ${(\Delta+\lambda-V^0)}$ to both sides of \eqref{eq:fred} and making use of proposition \ref{prop:neumann}, we can see that if $u_{sc}(\centerdot\,,y)$ solves \eqref{eq:fred}, then it solves the equation
    \[ (\Delta+\lambda-V)\,u_{sc}(\centerdot\,,y)=Vu_{in}(\centerdot\,,y)\textnormal{ in }\Rd. \] 
    Moreover, since 
    \[ u_{sc}(\centerdot\,,y)=(\Delta+\lambda+i0-V^0)^{-1}[(\gamma^s+\alpha\dd\sigma)\,u_{sc}(\centerdot\,,y)+V\,u_{in}(\centerdot\,,y)] \] and $(\gamma^s + \alpha\dd\sigma)\,u_{sc}(\centerdot\,,y)+V\,u_{in}(\centerdot\,,y)\in X_ \lambda$ is supported in $\D$, by proposition \ref{prop:neumann}, we can conclude that $u_{sc}(\centerdot\,,y)$ satisifes SRC \eqref{eq:src}.\\
    Now, we will make use of the Fredholm theory to solve the equation \eqref{eq:fred}. The Fredholm alternative theorem states that if $T$ a compact operator on a Banach space $\mathcal{B}$, then $(I-T)$ is invertible in $\mathcal{B}$ if and only if $(I-T)$ is injective. We will justify the use of this technnique in the following propositions by proving that, in essence, multiplication by $V^F\defeq \gamma^s+\alpha\dd\sigma$ defines a compact operator from $X_\lambda^*$ to $X_\lambda$. Start by proving the following:
    \begin{prop}\label{prop:restriction}
      For any  bounded open domain $\Omega\subset\Rd$, the restriction map
      \begin{align*}
        r_\Omega\,:\,X_\lambda^*&\longrightarrow H^1(\Omega)\\
        u&\longmapsto\restr{u}{\Omega}
      \end{align*}
      is a bounded operator.
    \end{prop}
    \begin{proof}
      Let $u\in X_\lambda^*$ and denote \[ u_I=\sum_{k\in I}P_k\,u,\quad u_{\mathbb{Z}\setminus I}=u-u_I. \]
      Note that $r_\Omega u= r_\Omega u_I + r_\Omega u_{\mathbb{Z}\setminus I}$, and therefore 
      \[ \norm{r_\Omega u}_{H^1(\Omega)}\leq \norm{r_\Omega u_I}_{H^1(\Omega)} + \norm{r_\Omega u_{\mathbb{Z}\setminus I}}_{H^1(\Omega)}. \]On the one hand, let $\alpha$ be a multiindex such that $|\alpha|\leq 1$. Then, by triangle, Hölder's and Bernstein's inequalities,
      \begin{align*}
        \norm{\partial^\alpha u_I}_{L^2(\Omega)}&\leq \sum_{k\in I} \norm{\partial^\alpha P_k\,u}_{L^2(\Omega)}\lesssim \sum_{k\in I} \norm{\partial^\alpha P_k\,u}_{L^{q_d}}\lesssim \sum_{k\in I} 2^{|\alpha|k} \norm{P_k\,u}_{L^{q_d}}\\
        &\lesssim \lambda^{\frac{1}{2}-\frac{d}{2}(\frac{1}{q_d}-\frac{1}{p_d})}\left(\sum_{k\in I} \lambda^{\frac{d}{2}(\frac{1}{q_d}-\frac{1}{p_d})}\norm{P_k\,u}_{L^{q_d}}^2\right)^{\half}\lesssim_\lambda \norm{u}_{X^*_\lambda},
      \end{align*}
      while
      \[ \norm{u_I}_{L^2(\Omega)}\leq\sum_{k\in I} \norm{P_k u}_{L^2(\Omega)}\leq\sum_{k\in I} \norm{P_k u}_{L^{q_d}}\lesssim_\lambda  \norm{u}_{X^*_\lambda}, \]
      so that
      \[ \norm{r_\Omega u_I}_{H^1(\Omega)}\lesssim_\lambda \norm{u}_{X^*_\lambda}. \]
      On the other hand, by Plancherel's identity and the triangle inequality,
      \begin{align*}
        \norm{u_{\mathbb{Z}\setminus I}}_{H^1(\Rd)}^2\leq & \norm{(I-\Delta)^{\half}\,u_{\mathbb{Z}\setminus I}}_{L^2}^2= \norm{(1+|\centerdot|^2)^{\half}\,\widehat{u_{\mathbb{Z}\setminus I}}}_{L^2}^2\\
        &\leq \norm{(1+|\centerdot|^2)^{\half}\,\widehat{P_{<I}\,u}}_{L^2}^2 + \sum_{k>k_\lambda+1}\norm{(1+|\centerdot|^2)^{\half}\,\widehat{P_k\,u}}_{L^2}^2\\
        &\lesssim_\lambda \norm{m_\lambda^{\half}\,\widehat{P_{<I}\,u}}_{L^2}^2 + \sum_{k>k_\lambda+1}\norm{m_\lambda^{\half}\,\widehat{P_k\,u}}_{L^2}^2\leq \norm{u}_{X_\lambda^*}^2,
      \end{align*} 
      where we have used the fact that, since \[ \supp(\widehat{P_{<I}\,u})\subset\{\xi\in\Rd\,:\,|\xi|\leq\lambda\} \] and, for $k>k_\lambda+1$, \[ {\supp(\widehat{P_{k}\,u})\subset\{\xi\in\Rd\,:\,2^{k-1}\leq|\xi|\leq2^{k+1}\}}, \] it follows that \[ (1+|\xi|^2)^{\half}\widehat{P_{<I}\,u}\sim m_\lambda^{\half}(\xi)\widehat{P_{<I}\,u}\quad \text{and}\quad (1+|\xi|^2)^{\half}\widehat{P_k\,u}\sim m_\lambda^{\half}(\xi)\widehat{P_k\,u}. \]
      Clearly, 
      \[ \norm{r_\Omega u_{I\setminus\mathbb{Z}}}_{H^1(\Omega)}\leq \norm{ u_{I\setminus\mathbb{Z}}}_{H^1(\Rd)}, \]
      which proves the proposition.
    \end{proof}
    \begin{cor}\label{cor:h1loc}
      Every $u\in X_\lambda^*$ belongs to $H^1_{\textnormal{loc}}(\Rd)$.
    \end{cor}
    Recall now that, if for any $s\geq 0$ we define $H^s(\Omega)$ as the space of restrictions of functions in $H^s(\Rd)$ to $\Omega$, we can then define $H^{-s}_0(\Omega)$ as the dual space of $H^s(\Omega)$.
    \begin{prop}\label{prop:inclusion}
      For any bounded open domain $\Omega\subset\Rd$, the embedding \[H^{-1}_0(\Omega)\lhook\joinrel\xrightarrow{\;i\;} X_\lambda\] is continuous.
    \end{prop}
    \begin{proof}
      We prove it by duality, using Hahn-Banach's Theorem. Indeed, let $\phi\in C^\infty_0(\Omega)$, then there exists $u\in X_\lambda^*$ such that \cite{brezisfunctional2011}
      \[
        \norm{\phi}_{X_\lambda}=\frac{\braket{u,\phi}}{\norm{u}_{X_\lambda^*}}\lesssim_\lambda \frac{\braket{u,\phi}}{\norm{r_\Omega\,u}_{H^1(\Omega)}} =\frac{\braket{r_\Omega\,u,\phi}}{\norm{r_\Omega\,u}_{H^1(\Omega)}}
        \leq \frac{\norm{r_\Omega\, u}_{H^1(\Omega)}\norm{\phi}_{H^{-1}_0(\Omega)}}{\norm{r_\Omega\,u}_{H^1(\Omega)}}=\norm{\phi}_{H^{-1}_0(\Omega)}, 
      \]
      where we have used proposition \ref{prop:restriction}. Now, $C^\infty_0(\Omega)$ is dense in $H^{-1}_0(\Omega)$ (proposition 2.9 in \cite{jerisoninhomogeneous1995}) and $X_\lambda$ is a Banach space \cite{caroscattering2020}, so the proof follows by a standard density argument.
    \end{proof}  
    \begin{prop}\label{prop:vf}
      Let $1/2<s<1$ and define $V^F\defeq \gamma^s+\alpha\dd\sigma$. There exists $C>0$ such that, for any $u,v\in\mathcal{S}(\Rd)$, it holds that
      \begin{equation}\label{eq:vf}
        \lvert\braket{\gamma^s u,\,v}\rvert\lesssim
        \norm{g}_{L^\infty}\left(\norm{u}_{H^s(\D)}\norm{v}_{L^{2}(\D)}+\norm{u}_{L^{2}(\D)}\norm{v}_{H^s(\D)}\right),
      \end{equation}
      where $\D$ is any open domain such that $\supp V\subset \D$.
      In particular, $V^F$ acts as a bounded multiplication operator from $H^s(\D)$ to $H^{-s}_0(\D)$.
    \end{prop}
    \begin{proof}
      We will use the homogeneous fractional Leibniz rule \cite{Christ1991, Kenig1993}. It is also known as Kato-Ponce differentiation rule, since its inhomogeneous version was first given by Kato and Ponce in \cite{katocommutator1988}. Indeed, for $u,v\in\mathcal{S}(\Rd)$, $1/r=1/p_1+1/q_1=1/p_2+1/q_2$ and $s<1$ one has
      \begin{equation}\label{eq:kato}
        \norm{D^s(uv)}_{L^r}\lesssim \norm{D^su}_{L^{p_1}}\norm{v}_{L^{q_1}}+\norm{u}_{L^{p_2}}\norm{D^sv}_{L^{q_2}}.
      \end{equation}
      In particular, taking $r=1$, $p_1=q_1=p_2=q_2=2$, one obtains
      \begin{align*}
        \lvert\braket{\gamma^s u,\,v}\rvert=&\lvert\braket{\chi^2D^s g,\,u v}\rvert=\lvert\braket{D^s g,\,\chi u\chi v}\rvert=\\&\lvert\braket{ g,\,D^s(\chi u\chi v)}\rvert\leq\norm{g}_{L^\infty}\norm{D^s(\chi u\chi v)}_{L^1}\lesssim\\
        &\norm{g}_{L^\infty}\left(\norm{D^s(\chi u)}_{L^{2}}\norm{\chi v}_{L^{2}}+\norm{\chi u}_{L^{2}}\norm{D^s(\chi v)}_{L^{2}}\right),
      \end{align*}
      where we have used Hölder's inequality and the aforementioned Kato-Ponce rule. On the one hand, since $\chi$ is supported in $\D$,
      \[ \norm{\chi u}_{L^2}\lesssim \norm{u}_{L^2(\D)}\lesssim \norm{u}_{H^s(\D)}. \]
      On the other hand, let any $\tilde{u}\in \mathcal{S}(\Rd)$ be such that $\restr{\tilde{u}}{\D}=\restr{u}{\D}$. Using again the fractional Leibniz rule we obtain
      \begin{align*}
        \norm{D^s(\chi \tilde{u})}_{L^{2}}\lesssim& \norm{D^s\chi}_{L^{\infty}}\norm{\tilde{u}}_{L^{2}}+\norm{\chi}_{L^{\infty}}\norm{D^s\tilde{u}}_{L^{2}}\lesssim \norm{\tilde{u}}_{H^s}.
      \end{align*}
      Now, since $\chi u=\chi \tilde{u}$, we can take infimum to obtain 
      \begin{equation}\label{eq:chi}
        \norm{D^s(\chi u)}_{L^{2}}\lesssim \inf \{\norm{\tilde{u}}_{H^s}:\restr{\tilde{u}}{\D}=\restr{u}{\D}\} = \norm{u}_{H^s(\D)}.
      \end{equation}
      Therefore,
      \begin{equation}\label{eq:vsigma}
        \lvert\braket{\gamma^s u,\,v}\rvert\lesssim
        \norm{g}_{L^\infty}\left(\norm{u}_{H^s(\D)}\norm{v}_{L^{2}(\D)}+\norm{u}_{L^{2}(\D)}\norm{v}_{H^s(\D)}\right)
      \end{equation}
      Note that in particular we have that
      \[ \lvert\braket{\gamma^s u,\,v}\rvert\lesssim \norm{g}_{L^\infty}\norm{u}_{H^s(\D)}\norm{v}_{H^s(\D)},  
        \]
      while for the $\alpha\dd\sigma$ term we have that
      \[ \abs{\braket{\alpha\dd\sigma u,\,v}}=\left|\int_{\Gamma}uv\alpha\dd\sigma\right|\leq\norm{\alpha}_{L^\infty(\Gamma)}\norm{u}_{L^2(\Gamma)}\norm{v}_{L^2(\Gamma)}. \]
      Observe that, by the trace theorem for Sobolev spaces (see, for example, \cite{triebeltheory1992}, section 4.4.2), for any $1/2<s<1$ it holds that
      \[ \norm{u}_{L^2(\Gamma)}\lesssim \norm{u}_{H^{s-1/2}(\Gamma)}\lesssim \norm{u}_{H^{s}(\D)}, \]
      so that
      \begin{equation}\label{eq:vfineq} 
        \lvert\braket{V^F u,\,v}\rvert\lesssim \left(\norm{g}_{L^\infty(\Rd)}+\norm{\alpha}_{L^\infty(\Gamma)}\right),
      \end{equation}
      which ends the proof of the proposition.
    \end{proof}
    \begin{prop}\label{prop:lr}
      The operator \[ (\Delta+\lambda+i0-V^0)^{-1}\circ(\gamma^s+\alpha\dd\sigma) \] is compact in $X^*_\lambda$.
    \end{prop}
    \begin{proof}
      By proposition \ref{prop:neumann}, $(\Delta+\lambda+i0-V^0)^{-1}$ is bounded from $X_\lambda$ to $X_\lambda^*$, so we only have to show that $V^F$ is compact form $X_\lambda^*$ to $X_\lambda$. All the necessary ingredients have been laid in the previous propositions. Just recall that the inclusion \[H^1(\D)\lhook\joinrel\xhookrightarrow{\;j\;}H^s(\D)\] is known to be compact and $H_0^{-s}(\D)$ is continuously embedded in $H_0^{-1}(\D)$. We summarize the argument in the following diagram:
      \[\begin{tikzcd}[sep=0pt]
        {X_\lambda^*} & {\xrightarrow{\;r_{\D}\;}} & {H^1(\D)} & {\lhook\joinrel\xhookrightarrow{\;j\;}} & {H^s(\D)} & {\xrightarrow{\; V^F\;}} & {H_0^{-s}(\D)} & {\xhookrightarrow{\;j^*\;}} & {H_0^{-1}(\D)} & {\xhookrightarrow{\;i\;}} & {X_\lambda}
      \end{tikzcd}\]
    \end{proof}
    Now, to prove existence of the scattering solution, we are only missing the injectivity of the operator \[ (\Delta+\lambda+i0-V^0)^{-1}\circ(\gamma^s+\alpha\dd\sigma). \] With the reader's consent, we will borrow a lemma that will be proved later in section \ref{sect:cgo}, namely lemma \ref{lem:carlemanest} concerning a Carleman estimate. This estimate works on a family of Bourgain spaces defined by the following norm, for $\zeta\in\mathbb{C}^d$ and $s\in\mathbb{R}$,
    \begin{equation}\label{eq:bourgain}
      \norm{u}_{X_\zeta^s}=\norm{(M|\Re(\zeta)|^2+M^{-1}|p_\zeta|^2)^{s/2}\,\widehat{u}}_{L^2},
    \end{equation}
    with $M>1$, where $\Re$ denotes the real part and 
\[ p_\zeta(\xi)=-|\xi|^2+2i\zeta\cdot\xi+\zeta\cdot\zeta. \] 
    \vspace{-15pt}
    \begin{restatable*}{lemma}{carlemanen}
      \label{lem:carlemanest}
      Let $R_0>0$ such that $\supp V\subset B_{R_0}=\{x\in\Rd:|x|<R_0\}$.  Take $\varphi_\zeta(x)=M\frac{(x\cdot\theta)^2}{2}+x\cdot\zeta,$ with $\zeta=\tau\theta+i\mathcal{I}$, $\theta\in\mathbb{S}^{d-1}$ and $\mathcal{I}\in\Rd$.
      There exists $C>0$ and $\tau_0=\tau_0(R_0,V,\lambda)$ such that
      \begin{equation}\label{eq:carlemanest}
        \norm{u}_{X^\half_{-\zeta}}\leq C R_0\norm{e^{\varphi_{\zeta}}(\Delta+\lambda-V)\,(e^{-\varphi_{\zeta}} u)}_{X_{-\zeta}^{-\half}}
      \end{equation}
      for all for $u\in\mathscr{S}(\Rd)$ with $\supp u\subset B_{R_0}$ and $\tau > \tau_0$.
    \end{restatable*}
    Take $R_0$ such that $\D\subset B_{R_0}$. Now, we can check that the spaces $X_\zeta^\half$ and $H^1(\Rd)$ are equal as sets, and that, for every $u\in H^1(\Rd)$ such that $\supp u\subset \overline{\D}$, we have that $e^{\varphi_{\zeta}}(\Delta+\lambda-V)\,(e^{-\varphi_{\zeta}} u)$ is in $X_{-\zeta}^{-\half}$. Therefore, by density, \eqref{eq:carlemanest} also holds for every $u\in H^1(\Rd)$ such that $\supp u\subset \overline{\D}$.
    \begin{lemma}
      Consider $d\geq 3$. If $u\in H^1_{\text{loc}}(\mathbb{R}^d)$ is a solution of \[ (\Delta+\lambda-V)\,u=0\text{ in } \Rd \] that satisfies the SRC \eqref{eq:src}, then $u$ has to be identically zero.
    \end{lemma}
    \begin{proof}
      Let $R>0$ and call $B=\{x\in\Rd\,:\,|x|<R\}$. On the one hand, the restriction of u to $\Rd\setminus\supp V$ solves $(\Delta+\lambda)\,u=0$. By Theorem 11.1.1 in \cite{hormanderlinear1963} this restriction is smooth, and we have that
      \begin{equation}\label{eq:srcuniq}
        \int_{\partial B} |\partial_\nu u-i\lambda^\half u|^2\,\dd S = \int_{\partial B}|\partial_\nu u|^2+\lambda |u|^2+i\lambda^\half (\partial_\nu u\,\overline{u}-\overline{\partial_\nu u}\,u)\dd S,
      \end{equation}
      where $\mathcal{I}$ denotes the imaginary part and $\partial_\nu=\nu\cdot\nabla$ the normal derivative with respect to the vector $\nu=x/|x|$. Using Green's identity in $B\setminus\overline{\D}$ we obtain that
      \begin{align*}
        \int_{\partial B} \partial_\nu u\,\overline{u}-\overline{\partial_\nu u}\,u\dd S=
        -\int_{\partial \D} \partial_\nu u\,\overline{u}-\overline{\partial_\nu u}\,u\dd S
      \end{align*}
      Now, by the SRC \eqref{eq:src}, identity \eqref{eq:srcuniq} yields
      \[ \lim_{R\to\infty}\int_{\partial B}|\partial_\nu u|^2+\lambda |u|^2=i\lambda^\half\int_{\partial \D} \partial_\nu u\,\overline{u}-\overline{\partial_\nu u}\,u\dd S. \]
      Since the potential $V$ is real-valued, Green's identity in $\D$ gives us that
      \[ i\int_{\partial \D} \partial_\nu u\,\overline{u}-\overline{\partial_\nu u}\,u\dd S=0, \]
      which implies that
      \[ \lim_{R\to\infty}\int_{\partial B}\lambda |u|^2=0, \]
      and, consequently, by Rellich's lemma, $\supp u\subset \overline{\D}$ and $u\in H^1(\Rd)$. Then, we can apply inequality \eqref{eq:carlemanest} to $v=e^{\varphi_\zeta}u$, which belongs to $H^1(\Rd)$ and is supported in $\overline{\D}$:
      \begin{equation*}
        \norm{e^{\varphi_{\zeta}}u}_{X^\half_{-\zeta}}\leq C R_0\norm{e^{\varphi_{\zeta}}(\Delta+\lambda-V)\, u}_{X_{-\zeta}^{-\half}},
      \end{equation*}
      where $R_0$ is such that $\D\subset B_{R_0}$. Finally, since $(\Delta+\lambda-V)\,u=0$, we can conclude that $u=0$.
    \end{proof}
    With this last lemma, the Fredholm theory argument is completed. Therefore, we have effectively proved the following theorem:
    \begin{thm:direct}
      Suppose $d\geq3$ and $V$ is of the form \eqref{V}. Then, there exists $\lambda_0=\lambda_0(V,d)$ such that, for every $\lambda\geq \lambda_0$, there is an unique solution $u_{sc}(\centerdot\,,y)\in X_\lambda^*$ to the problem \eqref{eq:helm} for every $y\in \mathbb{R}^d\setminus \supp V$.\\
      Moreover, the mapping $Vu_{in}(\centerdot\,,y)\mapsto u_{sc}(\centerdot\,,y)$ is bounded from $X_\lambda$ to $X_\lambda^*$.
  \end{thm:direct}
    \section{Runge approximation and Alessandrini identity}\label{sect:ortho}	
	In this section, we aim to prove the following orthogonality relation, which is crucial to prove the inverse uniqueness with local data result:
	\begin{restatable}{prop}{proportho}\label{proportho}
		Consider $d\geq 3$. Let $V_1$ and $V_2$ be two potentials of the form \eqref{V}, and let $\Sigma_1,\Sigma_2$ be two relatively open sets of dimension $d-1$, separated from $\supp V$, and that can be expressed as the graph of $\C^2$ functions. Choose a bounded open domain $\D$ of class $\C^2$ such that $\Sigma_1,\Sigma_2\subset\partial\D$, $\supp V_j\subset \D$, $j=1,2$ \\Then, there exists $\lambda_0=\lambda_0(V,d)$ such that, for every $\lambda\geq \lambda_0$ except for at most a countable set, it holds that \[ \restr{u_{sc,1}}{\Sigma_1\times\Sigma_2}=\restr{u_{sc,2}}{\Sigma_1\times\Sigma_2} \implies
		\braket{(V_1-V_2)v_1,v_2}=0,\] \[ \text{for all } v_1,v_2\in H^1(\D)\textrm{ such that }\left(\Delta +\lambda- V_j\right)v_j=0 \textrm{ in } \D. \]
	\end{restatable}
	We will divide the proof of this proposition in two parts. We will prove that the orthogonality relation is fullfiled for solutions that can be represented as single layer potentials with densities supported in $\Sigma_1$ and $\Sigma_2$ respectively, to then extend it to every solution by approximating them by these single layer potentials. We will start by providing with a Runge approximation result, that allows us to do this approximation in a $L^2$ norm. The ideas for the proof take inspiration from \cite{isakovuniqueness1988} and \cite{harrachmonotonicity2019}. However, due to the low-regularity of the potential $V$, we need an approximation in a space of higher regularity. Lemma \ref{lemma:int} below will provide us with an estimate to make the approximation in a $H^1$ norm. In this section, we will constantly refer to appendix \ref{ap:neumann}. There, we remember some classical results regarding layer potentials and their use in solving the Neumann problem for the Helmholtz equation, which is a key ingredient in our proof or the Runge approximation.
	\subsection{Runge approximation}\label{sect:runge}
  We will start the section by providing with two technical lemmas that make the proof of the Runge approximation more readable. The notation $\D'\subset\subset \D$ will mean here and throughout the paper that $\overline \D' \subset \D$. Also, note that, given the nature of $\gamma^s$, the integration has to be understood at times as a duality pairing.
  \begin{lemma}\label{lemma:runge1}
    Let $V$ be as in \eqref{V}. Let $\D'\subset\subset \D$ be two open bounded domains of class $\C^2$ such that $\supp V\subset \D'$, and let $v\in L^2(\D)$ be also supported in $\D'$. Suppose that $\varphi\in H^1(\D)$ is a solution to
    \begin{align*}
			\begin{cases}
				\;\helm\varphi=v &\text{in }\D,\\
				\;\partial_\nu \varphi=0 &\text{on }\partial \D.
			\end{cases}
		\end{align*}
    Then, for every $y\in \partial \D$, it holds that
     \begin{equation}\label{eq:fubuin}
      \int_{\partial \D}\varphi(x)\,\partial_{\nu_x}u_ {in}(x,y)\dd S(x)
     =-\int_{\D'}\left(\varphi(x)\,V(x)+v(x)\right)u_ {in}(x,y)\dd x.
     \end{equation}
  \end{lemma}
  \begin{proof}
    This lemma boils down to a careful integration by parts, analysing the singularity of $u_{in}(x,y)$ when $x=y$.
    Let $B_{y,\varepsilon}$ be the ball of radius $\varepsilon>0$ centered in $y\in\partial\D$. In appendix \ref{ap:neumann} we give the following expression:
    \begin{equation}\label{eq:limitingexp}
      \partial_{\nu_x}u_{in}(x,y)=G(y,x)|x-y|^{2-d},
    \end{equation} with $G$ a bounded function on $\partial \D\times \partial \D$. This means that $|\partial_{\nu_x}u_{in}|$ is a weakly singular kernel of order $d-2$, which in turn ensures the absolute integrability of $\varphi(\centerdot)\partial_{\nu_x}u_{in}(\centerdot\,,y)$ in $\partial \D$ (see lemma \ref{lem:kernel} in Appendix \ref{ap:neumann}). Therefore, by the Dominated Convergence Theorem and Green's formula, we have
    \begin{align*}
      \int_{\partial \D}\varphi(x)\,\partial_{\nu_x}u_ {in}(x,y)\dd S(x)=\lim_{\varepsilon\to 0}&\int_{\partial \D}(1-\mathbbm{1}_{B_{y,\varepsilon}}(x))\,\varphi(x)\,\partial_{\nu_x}u_ {in}(x,y)\dd S(x)\\
      =\lim_{\varepsilon\to 0}&\left(\int_{\partial \left(\D\setminus B_{y,\varepsilon}\right)}\varphi(x)\,\partial_{\nu_x}u_ {in}(x,y)\dd S(x)\right.\\
      &\left.-\int_{(\partial B_{y,\varepsilon})\cap \D}\varphi(x)\,\partial_{\nu_x}u_ {in}(x,y)\dd S(x)\right),
    \end{align*}
    First, we will see that the second term is negligible. Indeed, we can use again the limiting expression for $\partial_{\nu_x} u_{in}$ in \eqref{eq:limitingexp}, that is $|\partial_{\nu_x} u_{in}(x,y)|\lesssim |x-y|^{2-d}$ when $x$ approaches $y$. Changing to polar coordinates centered in $y$ we obtain, for $\varepsilon$ sufficiently small,
    \begin{align}\label{eq:limiting}
      \begin{split}
      \left|\int_{(\partial B_{y,\varepsilon})\cap \D}\varphi(x)\,\partial_{\nu_x}u_ {in}(x,y)\dd S(x) \right|\leq& \,\varepsilon\int_{\mathbb{S}^{d-1}_\cap}|\varphi(\varepsilon,\theta)|\dd S(\theta)\\
      \lesssim& \,\varepsilon \sup_{x\in \D\setminus \D'}|\varphi(x)|\xrightarrow{\varepsilon \to 0} 0,
      \end{split}
    \end{align}    
    where we have denoted by $\mathbb{S}^{d-1}_\cap$ the relevant half sphere in the change of variables, and we have used the fact that, since $\varphi$ solves $(\Delta+\lambda)\varphi=0$ in $\D\setminus \D'$, its restriction to that domain is smooth by Theorem 11.1.1 in \cite{hormanderlinear1963}, and therefore bounded.
    Next, we can use integration by parts in $\D\setminus B_{y,\varepsilon}$ to obtain
    \begin{align*}
      \int_{\partial \left(\D\setminus B_{y,\varepsilon}\right)}\varphi(x)\,\partial_{\nu_x}u_ {in}(x,y)\dd S(x)=&\int_{\D\setminus B_{y,\varepsilon}}\varphi(x)\helmx u_{in}(x,y)\dd x\\
			&-\int_{\D\setminus B_{y,\varepsilon}}\left(\Delta+\lambda-V(x)\right)\varphi(x)\,u_{in}(x,y)\dd x\\
      &+\int_{\partial \left(\D\setminus B_{y,\varepsilon}\right)}\partial_{\nu}\varphi(x)\,u_ {in}(x,y)\dd S(x).
    \end{align*}
    On the one hand, since $\partial_\nu \varphi=0$ on $\partial \D$, we have that
    $$
      \int_{\partial \left(\D\setminus B_{y,\varepsilon}\right)}\partial_{\nu}\varphi(x)\,u_ {in}(x,y)\dd S(x)= \int_{(\partial B_{y,\varepsilon})\cap \D}\partial_{\nu}\varphi(x)\,u_ {in}(x,y)\dd S(x).
    $$
    Now, observing that the limiting expression for $u_{in}$ is
    \begin{equation*}
      u_{in}(x,y)=F(y,x)|x-y|^{2-d},
    \end{equation*} with $F$ a bounded function on $\partial \D\times \partial \D$ (again refer to appendix \ref{ap:neumann}), we can do an identical argument as in \eqref{eq:limiting} to obtain that 
    $$
      \left|\int_{(\partial B_{y,\varepsilon})\cap \D}\partial_{\nu}\varphi(x)\,u_ {in}(x,y)\dd S(x).\right| \xrightarrow{\varepsilon\to 0} 0.
    $$
    On the other hand, we have that $\left(\Delta_x+\lambda\right)u_{in}(x,y)=0$ for $x\neq y$, while \newline${\left(\Delta+\lambda-V(x)\right)\varphi(x)=v(x)}$. 
    Therefore, since both $V$ and $v$ are supported in $\D'$,
    \begin{align*}
      \int_{\partial \D}\varphi(x)\,\partial_{\nu_x}u_ {in}(x,y)\dd S(x)=\lim_{\varepsilon\to 0}\int_{\D'}\left(-\varphi(x)\,V(x)
			-v(x)\right)u_{in}(x,y)\dd x.
    \end{align*}
    However, the integral in the limit does not depend on $\varepsilon$ in any way, and the lemma is proved.
  \end{proof}
  Remember now that total wave is $\uto = u_{in}+u_{sc}$, and consider the single layer potential with density $f\in\mathcal{C}(\partial\D)$ as defined by (see Appendix \ref{ap:neumann} for details)
	\begin{equation}\label{eq:slp}
    \left(\mathcal{S}f\right)(x)=\int_{\partial \D}f(y)\, \uto(x,y)\,\dd S(y), \quad x\in \Rd,
  \end{equation}
  and the normal derivative operator as
  \begin{equation}\label{eq:derslp}
    \left(\mathcal{N}f\right)(x)= \int_{\partial \D}f(y)\, \partial_{\nu_x}\uto(x,y)\dd S(y), \quad x\in \partial \D.
  \end{equation}
  \begin{lemma}\label{lemma:runge2}
    Let $V$ be as in \eqref{V}. Let $\D'\subset\subset \D$ be two open bounded domains of class $\C^2$ such that $\supp V\subset \D'$, and let $v\in L^2(\D)$ also supported in $\D'$. Suppose $\varphi\in H^1(\D)$ is a solution to
    \begin{align*}
			\begin{cases}
				\;\helm\varphi= v &\text{in }\D,\\
				\;\partial_\nu \varphi=0 &\text{on }\partial \D.
			\end{cases}
		\end{align*}
    Then, it holds that
    \[ \int_{\partial \D}(\mathcal{N}f)(x)\,\varphi(x)\dd S(x)=-\int_{\D'}(\mathcal{S}f)(x)\,v(x)\dd x. \]
  \end{lemma}
  \begin{proof}
    First note that scattering part $u_{sc}$ solves $\left(\Delta+\lambda\right)u_{sc}(\centerdot\,,y)=0$ in $\mathbb{R}^d\setminus \supp V$ as long as $y\in \mathbb{R}^d\setminus \supp V$, so it is smooth away from the potential, by Theorem 11.1.1 in \cite{hormanderlinear1963}, and in particular it is smooth on $\partial \D$. On the other hand, expression \eqref{eq:limitingexp} means that $|\partial_{\nu_x}u_{in}|$ is a weakly singular kernel of order $d-2$, which in turn ensures the absolute integrability of $\varphi(\centerdot)f(\circ)\partial_{\nu_x}u_{in}(\centerdot\,,\circ)$, and allows to use Fubini's Theorem (see lemma \ref{lem:kernel} in Appendix \ref{ap:neumann}). Therefore,
		\begin{align}\label{Tf}
		\begin{split}
			\int_{\partial \D}(\mathcal{N}f)(x)\,\varphi(x)\dd S(x)&=\int_{\partial \D}\varphi(x)\left(\int_{\partial \D}f(y)\,\partial_{\nu_x}\uto(x,y)\dd S(y)\right)\dd S(x)
			\\&=\int_{\partial \D}f(y)\left(\int_{\partial \D}\varphi(x)\,\partial_{\nu_x}\uto(x,y)\dd S(x)\right)\dd S(y).
		\end{split}
		\end{align}
		Now, for any $y\in\partial \D$ we obtain, by integrating by parts in $\D$ with respect to $x$, 
    \begin{align}\label{eq:fubusc}
      \begin{split}
			\int_{\partial \D}\varphi(x)\partial_{\nu_x}u_ {sc}(x,y)\dd S(x)=&\int_{\D}\varphi(x)\helmx u_{sc}(x,y)\dd x\\
			&-\int_{\D}\helm\varphi(x)\,u_{sc}(x,y)\dd x\\
      &+\int_{\partial \D}\partial_{\nu}\varphi(x)\,u_ {sc}(x,y)\dd S(x)\\
			=&\int_{\D'}\left(\varphi(x)\, V(x)\,u_{in}(x,y)-v(x)\,u_ {sc}(x,y)\right)\dd x.
      \end{split}
		\end{align}
		In the last identity we have used that $\helmx u_{sc}(\centerdot\,,y)=Vu_{in}(\centerdot\,,y)$ in $\mathbb{R}^d$, and also that $\helm\varphi=v$ in $\D$ and $\partial_\nu \varphi = 0$ on $\partial \D$, as well as the fact that both $V$ and $v$ are supported in $\D'$. Joining \eqref{eq:fubusc} above with \eqref{eq:fubuin} in lemma \ref{lemma:runge1} yields
		 $$
			\int_{\partial \D}\varphi(x)\,\partial_{\nu_x}\uto(x,y)\dd S(x) 
			= -\int_{\D'}v(x)\,u_{to}(x,y)\dd x.
		$$
		 Applying this identity in (\ref{Tf}) and using Fubini's Theorem again yields
		 \begin{align*}
		\int_{\partial \D}(\mathcal{N}f)(x)\,\varphi(x)\dd S(x)=& -\int_{\D'}v(x)\left(\int_{\partial \D}f(y)\,u_{to}(x,y)\dd S(y)\right)\dd x\\
		 =& -\int_{\D'}v(x)\mathcal (Sf)(x)\dd x.
		 \end{align*}
  \end{proof}
  Note that the single-layer potential $\mathcal{S}f$ as defined in \eqref{eq:slp} belongs to $H^1(\D)$ by lemma \ref{lem:single} in appendix \ref{ap:neumann}. Therefore, if we take two open bounded domains $\D'\subset\subset \D$ of class $\C^2$, and a relatively open subset $\Sigma\subset\partial \D$, we can define the following spaces:
	\begin{align*}
		&X=\left\{u\in H^1(\D)\,:\, \left(\Delta +\lambda- V\right)u=0 \textrm{ in } \D \right\}\\
		&X^\Sigma=\left\{u = \mathcal{S} f\in H^1(\D)\; : \; f\in \mathcal{C}(\partial \D), \;\supp f\subset \Sigma \right\}\\
		&X_{\D'}=\left\{\restr{u}{\D'}\,:\,u\in X \right\}\subset H^1(\D')\\
		&X^\Sigma_{\D'}=\left\{\restr{u}{\D'}\,:\,u\in X^\Sigma \right\}\subset H^1(\D')
	\end{align*}
	\begin{prop}[Runge approximation]\label{prop:runge}
		Let $V$ a potential of the form \eqref{V}. Let $\Sigma$ be a relatively open set of dimension $d-1$, separated from $\supp V$, and that can be expressed as the graph of $\C^2$ functions.\\
     Choose a bounded open domain $\D$ of class $\C^2$ such that $\Sigma\subset\partial\D$ and $\supp V\subset \D$, and let $\D'\subset\subset\D$ be a smooth domain such that $\supp V\subset \D'$, and such that $\D\setminus \overline{\D'}$ is connected.\\
		Then, there exists $\lambda_0=\lambda_0(V,d)$ such that, for every $\lambda\geq \lambda_0$ except for at most a countable set, $X^\Sigma_{\D'}$ is dense in $X_{\D'}$ under the $L^2(\D')$ norm, this is, any function in $X_{\D'}$ can be approximated by functions in $X^\Sigma_{\D'}$ in this norm.
	\end{prop}
    \begin{proof}
		We will show that $X_{\D'}\subset {({X_{\D'}^\Sigma})^\perp}^\perp=\overline{X_{\D'}^\Sigma}$, where the orthogonal complement and closure are taken with respect to the $L^2(\D')$ inner product. Indeed, let $v\in L^2(\D')$, and let $\varphi\in H^1(\D')$
		such that 
		\begin{align}
			\begin{cases}
				\;\helm\varphi=E_0v &\text{in }\D,\\
				\;\partial_\nu \varphi=0 &\text{on }\partial \D,
			\end{cases}
			\label{2}
		\end{align} where $E_0v$ denotes the extension by $0$ of $v$ to $\D$. This $\varphi$ will be guaranteed to exist as long as $\lambda$ is not a Neumann eigenvalue for the operator $-\Delta+V$ in $\D$. The set of these eigenvalues is a countable set. For a proof of these claims, see appendix \ref{ap:neumann}. Now, we have that
		\begin{align*}
			v\in ({X_{\D'}^\Sigma})^\perp \iff 0=\braket{r_{\D'} u,v}_{\D'}=\braket{u,E_0v}_{\D}=\braket{u,\helm \varphi}_{\D}, \quad \forall u\in X^\Sigma,
		\end{align*}
		Applying Green's identity, and the fact that $\helm u = 0$ in $\D$ if $u\in X^\Sigma$,
		\begin{align*}
			v\in ({X_{\D'}^\Sigma})^\perp \iff 0=&\braket{u,\helm \varphi}_{\D}\\
			=& \braket{\helm u,\varphi}_{\D}+\braket{u,\partial_\nu\varphi}_{\partial \D}-\braket{\partial_\nu u,\varphi}_{\partial \D}, \quad \forall u\in X^\Sigma.\\
			\iff0=&\braket{\partial_\nu u,\varphi}_{\partial \D},  \quad \forall u\in X^\Sigma.
		\end{align*}
		Now, since $u = \mathcal{S}f$, we have that (see lemma \ref{lem:single}), for $x\in \partial \D$, \[ \partial_\nu u (x)= \frac{1}{2}f (x)+ \left(\mathcal{N}f\right)(x), \] with the operator $\mathcal{N}$ as defined in \eqref{eq:derslp}.
		Therefore,
		$$
			v\in ({X_{\D'}^\Sigma})^\perp \iff 0=\frac{1}{2}\braket{f,\varphi}_{\partial \D}+\braket{\mathcal{N}f,\varphi}_{\partial \D}, \quad \forall f\in \mathcal{C}(\partial \D)\, \text{ s.th. }\supp f\subset\Sigma.\\
		$$
    Now, by lemma \ref{lemma:runge2}, we have that, for $f\in \mathcal{C}(\partial \D)$ supported on $\Sigma$,
    \[ \braket{\mathcal{N}f,\varphi}_{\partial \D}=-\braket{\mathcal{S} f,v}_{\D'}=0, \]
    since $v\in ({X_{\D'}^\Sigma})^\perp$ by assumption and $\restr{\mathcal{S} f}{\D'}\in X_{\D'}^\Sigma$ by definition. This yields
		\[ v\in ({X_{\D'}^\Sigma})^\perp \iff 0=\braket{f,\varphi}_{\partial \D}, \quad \forall f\in \mathcal{C}(\partial \D)\, \text{ s.th. }\supp f\subset\Sigma.\\ \]
		Consequently,
		\begin{equation}\label{phicond}
			v\in ({X_{\D'}^\Sigma})^\perp \iff \varphi=0 \quad\textrm{on } \Sigma
		\end{equation}
	   Then, $\varphi$ solves
	   $$\begin{cases}
		   \;(\Delta+\lambda)\,\varphi=0 &\text{in }\D\setminus\overline{\D'},\\
		   \;\partial_\nu \varphi=\varphi=0 &\text{on }\Sigma,
	   \end{cases}$$
	   which, by the Unique Continuation property for the operator $(\Delta-\lambda)$ \cite{hormanderuniqueness1983}, implies that $\varphi=0$ in $\D\setminus\overline{\D'}.$ Also, since $\varphi$ solves $(\Delta+\lambda)\varphi=v$ in $\D\setminus \supp V$, its restriction to this set belongs to $H^2(\D\setminus \supp V)$ which means that $\varphi=\partial_\nu\varphi=0$ on ${\partial\D'\subset \D\setminus\supp V}$.\par
	   Finally, let $w\in X_{\D'}$. We want to show that $w\in {\left(({X_{\D'}^\Sigma})^\perp\right)}^\perp$. Indeed, let ${v\in ({X_{\D'}^\Sigma})^\perp}$ and $\varphi$ as in \eqref{2}. Then, using Green's Identity again,
	   \begin{align*}
		   \braket{w,v}_{\D'}=&\braket{w,\helm\varphi}_{\D'}\\
		   =& \braket{\helm w,\varphi}_{\D'}+\braket{w,\partial_\nu\varphi}_{\partial \D'}-\braket{\partial_\nu w,\varphi}_{\partial \D'}\\
		   =& \braket{w,\partial_\nu\varphi}_{\partial \D'}-\braket{\partial_\nu w,\varphi}_{\partial \D'}=0,
	   \end{align*}
	   where we have used that $\helm w=0$ in $\D'$. This proves that $X_{\D'}\subset \overline{X_{\D'}^\Sigma}$, which ends the proof.
	\end{proof} 
    As we commented above, we need to make the approximation of proposition \ref{prop:runge} in a stronger norm than $L^2$. This will be clear later in the proof of proposition \ref{proportho}, but the main reason is that the potential $V$ doesn't act as a bilinear operator over $L^2$, but over $H^1$. The next lemma is an interior regularity result, also called Caccioppoli inequality in the literature \cite{giaquintaintroduction2012}. It will allow us to get $H^1$ convergence in a smaller set than the one in which we have proven the $L^2$ approximation. Both the statement and the proof of the lemma are inspired by \cite{chensecond1998}. However, due to the nature of our potential $V$, we will obtain less regularity, and we will need to use both Sobolev embeddings and interpolation of Sobolev spaces to complete the proof.
	\begin{lemma}[Interior regularity]\label{lemma:int}
	Consider $d\geq 3$. Let V be a potential of the form \eqref{V}. Let $\Omega'\subset\subset\D$ be two Lipschitz open domains such that $\supp V\subset\Omega'$. Then, if $u\in H^1(\D)$ solves
	\begin{equation}
		\helm u = 0\quad \textrm{in }\D,\label{eqhelm}
	\end{equation}
  then it holds that
	\[ \norm{u}_{H^1(\Omega')}\lesssim \norm{u}_{L^2(\D)}. \]
	\end{lemma}
	\begin{proof}
        We will assume that $u$ is a real function. For a complex $u$, both its real and imaginary parts are real functions that satisfy (\ref{eqhelm}), since $V$ is real, thus the result will follow.\\ 
        The weak formulation of equation (\ref{eqhelm}) is
        \begin{equation}
            \int_{\D}\nabla u \nabla \phi = \int_{\D}(V-\lambda)\,u\,\varphi, \quad \forall \varphi\in H_0^1(\D).\label{weakform}
        \end{equation}
        Consider now a real cut-off function $\eta$ between $\Omega'$ and $\partial\D$, namely, $\eta\in \mathcal{C}^\infty_0(\D;\mathbb{R})$ such that $\eta\equiv 1$ in $\Omega'$ and $0\leq\eta\leq 1$ in $\D\setminus\Omega'$. Consider then the test function $\varphi=\eta^2u\in H_0^1(\D)$. With this test function, eq. (\ref{weakform}) becomes
        \[ \int_{\D}\abs{\nabla u}^2\eta^2+\int_{\D}\nabla u\,2\eta\nabla\eta\, u=\int_{\D}(V-\lambda)\,u^2\eta^2. \]
        Therefore, since $\supp V\subset \Omega'$,
        \begin{equation}\label{desint}
            \int_{\D}\abs{\nabla u}^2\eta^2\leq \abs{\int_{\D}\nabla u\,2\eta\,u\nabla\eta}+{\abs{\int_{\D}\lambda\,u^2\eta^2}}+\abs{\int_{\Omega} V\,u^2}.
        \end{equation}
        Now, let's analyze these three summands separately. For the second one we have, since $0\leq\eta\leq1$,
        \begin{equation}\label{eq:di1}
          \abs{\int_{\D}\lambda\,u^2\eta^2}\leq \lambda\int_{\D}u^2= \lambda\norm{u}^2_{L^2(\D)}.
        \end{equation}
        For the first one, set any $0<\varepsilon<1$. By Young's inequality for products, we have that $2ab\leq\varepsilon a^2+\varepsilon^{-1}b^2$ for any two $a,b\geq0$. This, along with Cauchy-Schwartz inequality gives
        \begin{align}\label{eq:di2}
          \begin{split}
          \abs{\int_{\D}\nabla u\,2\eta\,u\nabla\eta}\leq &\,\varepsilon \int_{\D}\abs{\nabla u}^2\eta^2 + \varepsilon^{-1}\int_{\D}\abs{\nabla \eta}^2u^2\\
          \leq& \,\varepsilon \int_{\D}\abs{\nabla u}^2\eta^2 + \varepsilon^{-1}\norm{\nabla \eta}^2_{L^\infty(\D)}\norm{u}^2_{L^2(\D)}.
          \end{split}
        \end{align}
        Finally, in the last summand of \eqref{desint} we will analyze $V^0$ and $V^F=\alpha\dd\sigma+\gamma^s$ separately:
        \[ \abs{\int_{\Omega'} V\,u^2}\leq\abs{\int_{\Omega'}V^0u^2}+\abs{\int_{\Omega'}V^F u^2}. \]
        To estimate the second term, use \eqref{eq:vfineq} in the proof of proposition \ref{prop:vf}. Indeed, there exists $K>0$ such that
        \[ \abs{\int_{\Omega'}V^F u^2}\leq K \left(\norm{g}_{L^\infty(\Rd)}+\norm{\alpha}_{L^\infty(\Gamma)}\right)\norm{u}_{H^s(\Omega')}^2. \]
        Now, note that in a domain with smooth boundary, the fractional order Sobolev space $H^s(\Omega')$ is equivalent to the interpolation space between $L^2(\Omega')$ and $H^1(\Omega')$ with interpolation index $s$ (see, for example, \cite{triebelinterpolation1978}), that is
        \[ \norm{u}^2_{H^s(\Omega')}\leq \norm{u}^{2(1-s)}_{L^2(\Omega')}\norm{u}^{2s}_{H^1(\Omega')}. \]
        Also, if we take $\delta>0$, using again Young's inequality for products, we have that $ab\leq \delta^{-p}\frac{a^p}{p}+\delta^{p'}\frac{b^{p'}}{p'}$, for any $a,b\geq0$ and $\frac{1}{p}+\frac{1}{p'}=1$. Choosing $p=\frac{1}{1-s}$, $p'=\frac{1}{s}$ gives
        \[ \norm{u}^2_{H^s(\Omega')}\leq (1-s)\delta^{-\frac{1}{1-s}}\norm{u}^{2}_{L^2(\Omega')}+s\delta^{1/s}\norm{u}^{2}_{H^1(\Omega')}. \]
        Therefore, since $\norm{u}_{L^2(\Omega')}\leq\norm{u}_{L^2(\D)}$, we have that
        \begin{align}\label{eq:di3}
            \abs{\int_{\Omega'}V^F u^2}
            \leq K\left(\norm{g}_{L^\infty(\Rd)}+\norm{\alpha}_{L^\infty(\Gamma)}\right)\left[(1-s)\delta^{-\frac{1}{1-s}}\norm{u}^{2}_{L^2(\D)}+
            s\delta^{1/s}\norm{u}^{2}_{H^1(\Omega')}\right].
        \end{align}
        Now, to estimate the $V^0$ term, let $N>0$ and consider the set \[F=\{x\in\Rd\,:\,|V^0(x)|>N\},\] and define $E=\Rd\setminus F$. Then, $V^0=\mathbbm{1}_E V_0+\mathbbm{1}_F V_0$, with $\norm{\mathbbm{1}_E V_0}_{L^\infty}=N$.
        Then, using Hölder inequality and the Sobolev embedding $H^1(\Omega')\subset L^{p_d}(\Omega')$,
        \begin{align}\label{eq:di4}
          \begin{split}
            \abs{\int_{\Omega'}V^0u^2}\leq&\abs{\int_{\Omega'}V^0\mathbbm{1}_Eu^2}+\abs{\int_{\Omega'}V^0\mathbbm{1}_Fu^2}\\
            \leq& N\norm{u}^2_{L^2(\Omega')}+\norm{V^0\mathbbm{1}_F}_{L^{d/2}(\Omega')}\norm{u}^2_{L^{p_d}(\Omega')}\\
            \leq& N\norm{u}^2_{L^2(\Omega')}+C\norm{V^0\mathbbm{1}_F}_{L^{d/2}(\Omega')}\norm{u}^2_{H^1(\Omega')},
          \end{split}
        \end{align}
        where $C$ is the Sobolev embedding constant. Putting \eqref{eq:di1}, \eqref{eq:di2}, \eqref{eq:di3} and \eqref{eq:di4} in inequality \eqref{desint} yields
        \[ (1-\varepsilon)\int_{\D}\abs{\nabla u}^2\eta^2\leq C_1\norm{u}^2_{L^2(\D)}+C_2\norm{u}^2_{H^1(\Omega')}, \]
        where
        \begin{align*}
        C_1&=\lambda+N+\varepsilon^{-1}\norm{\nabla\eta}^2_{L^\infty(\D)}+\left(\norm{g}_{L^\infty(\Rd)}+\norm{\alpha}_{L^\infty(\Gamma)}\right)(1-s)\delta^{-\frac{1}{1-s}},\\
        C_2&=K\left(\norm{g}_{L^\infty(\Rd)}+\norm{\alpha}_{L^\infty(\Gamma)}\right)s\delta^{1/s}+C\norm{V^0\mathbbm{1}_F}_{L^{d/2}(\Omega')}.
        \end{align*}
        Therefore we have, again since $0\leq\eta\leq 1$,
        \[ (1-\varepsilon)\norm{u}^2_{H^1(\Omega')}\leq C_1\norm{u}^2_{L^2(\D)}+C_2\norm{u}^2_{H^1(\Omega')}, \]
        so that
        \[ (1-\varepsilon-C_2)\norm{u}^2_{H^1(\Omega')}\leq C_1\norm{u}^2_{L^2(\D)}. \]
        Finally, since we can take $\varepsilon$ and $\delta$ as small as we want, and we can take $N$ as big as needed to make $\norm{V^0\mathbbm{1}_F}_{L^{d/2}}$ sufficiently small so that
        \[ C_2+\varepsilon<1. \] Therefore, we can conclude that
        \[ \norm{u}_{H^1(\Omega')}\lesssim \norm{u}_{L^2(\D)}, \]
        and the lemma is proved.
        \end{proof}

        \subsection{Proof of proposition \ref{proportho}}\label{sect:proportho}
	
        To prove proposition \ref{proportho} we will further need the following lemmas, whose proof can be again be found in \cite{caroscattering2020}. It might be of interest to note that, although we introduce a wider class of potentials, the proof is based on working away from its support, and therefore it goes through identically.
        \begin{lemma}[\cite{caroscattering2020}]\label{lemmareci}
        Let $\D$ be a bounded open domain of class $\C^2$. The scattering solution of \eqref{eqhelm} satisfies the following reciprocity relation
        \[ u_{sc}(x,y)=u_{sc}(y,x), \quad \forall x,y\in \mathbb{R}^d\setminus \supp V. \]
        In particular, the single layer potential $\mathcal{S}$ is symmetric, that is,
        \[ \int_{\partial \D}\mathcal{S}f\,g\,\dd S=\int_{\partial \D}f\,\mathcal{S}g\,\dd S, \quad \forall f,g\in\mathcal{C}(\partial \D). \]
        \end{lemma} 
        The following lemma is also proved in the appendix, lemma \ref{lem:single}. From now on, when necessary, we will denote by $u_+$ the trace on $\partial \D$ of $\restr{u}{\Rd\setminus \overline{\D}}$, and by $u_-$ the trace on $\partial \D$ of $\restr{u}{\D}$. As well, we will denote by $\partial_\nu u_+$ and $\partial_\nu u_-$ the normal derivative of those, always with respect to the outward-pointing normal vector of $\partial \D$ (as seen from inside $\D$).
        \begin{lemma}[\cite{caroscattering2020}]\label{lemmaslp}
        Consider $d\geq3$, and let $\D$ be a bounded open domain of class $\C^2$. Let $f\in\mathcal{C}(\partial \D)$. Then, $u=\mathcal{S}f$ is the unique solution in $H^1_{\text{loc}}(\mathbb{R}^d)$ to the problem
        \begin{equation*}
            \begin{cases}
            \helm u = 0 & \textrm{in }\mathbb{R}^d\setminus \partial \D,\\
            \partial_\nu u_--\partial_\nu u_+ = f & \textrm{on }\partial \D,\\ 
            u\textrm{ satisfying SRC}.
            \end{cases}
        \end{equation*}
        \end{lemma}
        After these considerations, we will proceed by first proving that the orthogonality relation is fullfiled by solutions represented as single layer potentials with densities supported in $\Sigma_1$ and $\Sigma_2$. The proof is similar to that of proposition 3.3 in \cite{caroscattering2020}.
        \begin{lemma}
          Consider $d\geq 3$. Let $V_1$ and $V_2$ be two potentials of the form \eqref{V}, and let $\Sigma_1,\Sigma_2$ be two relatively open sets of dimension $d-1$, separated from $\supp V$, and that can be expressed as the graph of $\C^2$ functions. Choose a bounded open domain $\D$ of class $\C^2$ such that $\Sigma_1,\Sigma_2\subset\partial\D$ and $\supp V_j\subset \D$,  $j=1,2$,.\\ Then, there exists $\lambda_0=\lambda_0(V,d)$ such that, for every $\lambda\geq \lambda_0$, it holds that \[ \restr{u_{sc,1}}{\Sigma_1\times\Sigma_2}=\restr{u_{sc,2}}{\Sigma_1\times\Sigma_2} \implies
		    \braket{(V_1-V_2)v_1,v_2}=0,\] \[ \text{for all } v_1,v_2\in H^1(\D)\textrm{ such that }v_j=\mathcal{S}_jf_j,\textrm{ for some }f_j\in\mathcal{C}(\partial \D) \textrm{ supported in } \Sigma_j.\]  
        \end{lemma}\label{lemma:rungeorto}
        \begin{proof}
          Integrating in $\D$ and using Green's identity gives
          \begin{equation}\label{v1v2b0}
          \braket{(V_1-V_2)v_1,v_2}=\int_{\partial \D} (v_2\partial_\nu {v_1}_--v_1\partial_\nu {v_2}_-)\,\dd S,
          \end{equation}
          while doing so in $B\setminus \overline{\D}$, where $B\defeq\{x\in\mathbb{R}^d\,:\,\abs{x}<R\}$ gives
          \begin{align*}
              0=&\int_{B\setminus \overline{\D}}(\Delta+\lambda)v_1v_2=\\=&\int_{\partial B} (v_2\partial_\nu v_1-v_1\ndev v_2)\,\dd S-\int_{\partial \D} (v_2\partial_\nu {v_1}_+-v_1\ndev {v_2}_+)\,\dd S,
          \end{align*}
          since $v_1$ and $v_2$ are solutions to $(\Delta+\lambda)\,u=0$ in $\Rd\setminus\overline{\D}$ by lemma \ref{lemmaslp}.
          Here, making $R\to\infty$ and applying SRC \eqref{eq:src} yields
          \begin{equation}\label{itszero}
          \int_{\partial \D} (v_2\partial_\nu {v_1}_+-v_1\ndev {v_2}_+)\,\dd S=0.
          \end{equation}
          Then, inserting \eqref{itszero} in \eqref{v1v2b0},
          \begin{equation}
          \braket{(V_1-V_2)v_1,v_2}=\int_{\partial \D} \left[v_2\,(\ndev{v_1}_- -\ndev{v_1}_+)-v_1\,(\ndev{v_2}_--\partial_\nu {v_2}_+)\right]\,\dd S.
          \end{equation}
          Now, since $v_j=\mathcal{S}_j\,f_j$, applying lemma \ref{lemmaslp},
          \[ \braket{(V_1-V_2)v_1,v_2}=\int_{\partial \D} (\mathcal{S}_2 f_2\,f_1-\mathcal{S}_1 f_1\, f_2)\,\dd S, \]
          and, by lemma \ref{lemmareci},
          \[ \braket{(V_1-V_2)v_1,v_2}=\int_{\partial \D} \left[\mathcal{S}_2-\mathcal{S}_1\right] f_1\,f_2\,\dd S. \]
          Finally, since $u_{sc,1}=u_{sc,2}$ in $\Sigma_1\times\Sigma_2$ by assumption, the kernel of $\mathcal{S}_2-\mathcal{S}_1$ is supported in $(\partial \D\times\partial \D)\setminus(\Sigma_1\times\Sigma_2)$, but $\supp (f_1\otimes f_2)\subset \Sigma_1\times\Sigma_2$, so we conclude that 
          \[ \braket{(V_1-V_2)v_1,v_2}=0, \]
          and the lemma is proved. 
        \end{proof}
        Now, we are set for the proof of the orthogonality relation. We restate the proposition again:
        \proportho*
        \begin{proof}
            Now let $v_1,v_2\in H^1(\D)$ be such that to $\left(\Delta +\lambda- V_j\right)v_j=0 \textrm{ in } \D$. Let $\Omega''$ and $\D'$ be two open domains such that \[ \supp V_j\subset \Omega''\subset\subset \D'\subset\subset \D,\quad j=1,2, \]
            and let $(v_j^{(m)})_{m\in\mathbb{N}}\subset X^{\Sigma_j}$, $j=1,2$, be two sequences such that $v_j^{(m)}\xrightarrow{L^2(\D')}v_j$, which exist in virtue of proposition \ref{prop:runge}, meaning that $v_j^{(m)}=\mathcal{S}_j f_j^{(m)}$ for some $f_j^{(m)}$ supported in $\Sigma_j$. Then,
            \begin{align}\label{sum}
            \begin{split}
            \braket{(V_1-V_2)v_1,v_2}=&\braket{(V_1-V_2)(v_1-v_1^{(m)}),v_{2}}+\braket{(V_1-V_2)v_1^{(m)},v_2-v_2^{(m)}}\\&+\braket{(V_1-V_2)v_1^{(m)},v_2^{(m)}}.
            \end{split}
            \end{align}
            By lemma \ref{lemma:rungeorto}, we have that 
            \begin{equation}\label{sum3}
            \braket{(V_1-V_2)v_1^{(m)},v_2^{(m)}}=0,\quad \forall m\in\mathbb{N}.
            \end{equation} 
            We will now show that the two first summands converge to zero as $m\to\infty$. Note first that, for $\phi,\varphi\in H^1(\Omega'')$, it holds that
            \[ \braket{V^0_j\,\phi,\varphi}\leq \norm{V^0_j}_{L^{d/2}}\,\norm{\phi}_{L^{p_d}(\Omega'')}\,\norm{\varphi}_{L^{p_d}(\Omega'')}\lesssim \norm{\phi}_{H^1(\Omega'')}\norm{\varphi}_{H^1(\Omega'')}, \]
            and
            \[ \braket{V^F\,\phi,\varphi}\lesssim\left(\norm{g}_{L^\infty(\Rd)}+\norm{\alpha}_{L^\infty(\Gamma)}\right)\norm{\phi}_{H^s(\Omega'')}\,\norm{\varphi}_{H^s(\Omega'')}\lesssim \norm{\phi}_{H^1(\Omega'')}\norm{\varphi}_{H^1(\Omega'')}. \] 
            We have used in the estimates above the Hardy-Littlewood-Sobolev inequality and inequality \eqref{eq:vfineq} in proposition \ref{prop:vf}.
             Therefore, lemma \ref{lemma:int} yields
             \begin{equation}\label{eq:sum1}
                \braket{(V_1-V_2)(v_1-v_1^{(m)}),v_{2}} \lesssim \norm{v_1-v_1^{(m)}}_{H^1(\Omega'')}\,\norm{v_2}_{H^1(\Omega'')} \lesssim \norm{v_1-v_1^{(m)}}_{L^2(\D')}.
             \end{equation}
            In the second summand we proceed as in the first one. Note that, since $v_1^{(m)}$ are also solutions in $H^1(\D)$ of $\helm u = 0$, by lemma \ref{lemma:int},
            \[ \norm{v_1^{(m)}}_{H^1(\Omega'')}\leq \norm{v_1^{(m)}}_{L^2(\D')}, \quad \forall m\in\mathbb{N}, \]
            and, since $(v_1^{(m)})_{m\in\mathbb{N}}$ is a convergent sequence, it must also be bounded. This is, $\exists C\geq0$ such that
            \[ \norm{v_1^{(m)}}_{L^2(\D')}\leq C, \quad \forall m\in\mathbb{N}. \]
            Thus, we obtain
            \begin{equation}
                \label{sum2}\lvert\braket{(V_1-V_2)v_1^{(m)},v_2-v_2^{(m)}}\rvert\lesssim\norm{v_2-v_2^{(m)}}_{L^2(\D')}.
            \end{equation}
            By putting \eqref{sum3}, \eqref{eq:sum1} and \eqref{sum2} in \eqref{sum} we obtain
            \begin{equation*}
            \lvert\braket{(V_1-V_2)v_1,v_2}\rvert\lesssim \norm{v_1-v_1^{(m)}}_{L^2(\D')}+\norm{v_2-v_2^{(m)}}_{L^2(\D')}\xrightarrow{m\to\infty}0,
            \end{equation*}
            which proves our statement.
        \end{proof}

        \section{CGO solutions and proof of theorem \ref{thm:1.2}} \label{sect:cgo}
        Now that we have proved the Alessandrini-type identity in proposition \ref{proportho}, we can test it with the CGO solutions. These will be solutions of the form \begin{equation}\label{eq:cgo}v_j(x)=e^{\zeta_j\cdot x}(1+w_j(x)),\end{equation}
        for some $\zeta_j\in\mathbb{C}^d$. We will follow once again the arguments in \cite{caroscattering2020} and \cite{caroglobal2016} for this section. In particular, applying the operator $(\Delta+\lambda-V)$ in \eqref{eq:cgo} yields
        \[ (\Delta+2\zeta_j\cdot\nabla+\zeta_j\cdot\zeta_j+\lambda-V)w_j=V-\lambda-\zeta_j\cdot\zeta_j. \]
        Now if $\zeta_j$ are chosen such that $\zeta_j\cdot\zeta_j=-\lambda$, we obtain
        \[ (\Delta+2\zeta_j\cdot\nabla-V)w_j=V. \]
        Therefore, to prove the existence of these solutions, it is enough to prove injectivity of the formal adjoint $(\Delta-2\zeta_j\cdot\nabla-V)$. This can be done via a priori estimates. This estimate will be proved in section \ref{sect:carleman} in the relevant spaces, based on a Carleman estimate for the laplacian by Caro and Rogers \cite{caroglobal2016}, that can be perturbed to include the potentials $V$.  The inequalities will be analogous to those in \cite{caroscattering2020} and \cite{caroglobal2016} but, besides adding the potential $\gamma^s$, we will follow a slightly different order in rotating the inequalities and adding the potentials.\\
        Later, in section \ref{sect:inversethm} we will end the proof of Theorem \ref{thm:1.2}. We will further choose $\zeta_j$ to fulfill $\zeta_1+\zeta_2=-i\kappa$ for an arbitray $\kappa\in\Rd$ -which is possible in dimension $d\geq3$-, and the correction term $w_j$ will vanish in a certain sense when $|\zeta_j|$ grows. Here the approach will be based in \cite{caroscattering2020,caroglobal2016,habermanuniqueness2015}.

        \subsection{Existence of CGO solutions}\label{sect:carleman}

        As we mentioned in the introduction, we will prove the existence of CGO solutions via a priori estimates. For $s\in\mathbb{R}$ and $\zeta\in\mathbb{C}^d$ we define the inhomogeneous Bourgain space $X_\zeta^s$ as the space of distributions $u\in\mathscr{S}'(\Rd)$ such that $\widehat{u}\in L^2_\text{loc}(\Rd)$ and 
\[ \norm{u}_{X_\zeta^s}=\norm{(M|\Re(\zeta)|^2+M^{-1}|p_\zeta|^2)^{s/2}\,\widehat{u}}_{L^2}<\infty, \]
endowed with the norm $\norm{\centerdot}_{X_\zeta^s}$, with $M>1$, where $\Re$ denotes the real part and 
\[ p_\zeta(\xi)=-|\xi|^2+2i\zeta\cdot\xi+\zeta\cdot\zeta. \] 
These spaces were originally considered by Haberman and Tataru in \cite{habermanuniqueness2013}, then by Haberman in \cite{habermanuniqueness2015}, by Caro and Rogers in \cite{caroglobal2016}, and by Caro and García in \cite{caroscattering2020}. To prove the a priori estimate, we draw from Theorem 2 in \cite{caroglobal2016}. From now on, set $\zeta=\tau\theta+i\mathcal{I}$, with $\tau>0$, $\theta\in\mathbb{S}^{d-1}$, $\mathcal{I}\in\Rd$, and $\zeta\cdot\zeta\leq 0$, and define 
\[ \varphi_\zeta(x)=M\frac{(x\cdot\theta)^2}{2}+x\cdot\zeta. \]
Then, Theorem 2 in \cite{caroglobal2016} is roughly equivalent to the following:
\begin{theorem}[\cite{caroglobal2016}, Theorem 2]\label{thm:cgo}
Take $R_0>0$. There is an absolute constant $C>0$ such that, if $M>CR_0^2$ then, 
\[ \norm{u}_{X_{-\tau e_n}^\half}\leq C R_0\norm{e^{\varphi_{\tau e_n}}\Delta(e^{-\varphi_{\tau e_n}} u)}_{X_{-\tau e_n}^{-\half}} \]
for $u\in\mathscr{S}(\Rd)$ with $\supp u\subset B_{R_0}=\{x\in\Rd:|x|<R_0\}$ and $\tau > 8MR_0$.
\end{theorem}
We will transform the inequality to make it work for an arbitrary $\zeta$. If for $\zeta=\tau\theta+i\mathcal{I}$ as above we take $Q\in SO(d)$ to be a rotation such that $Qe_n=\theta$, and denote by $Q^T$ its traspose, and by $Q_*$ its pullback, then it is easy to check that $\varphi_{\tau e_n}=Q_*\varphi_{\tau\theta}$ and $p_{-\tau e_n}=Q_*p_{-\tau\theta}$. Thus, if $u\in\mathscr{S}(\Rd)$ with $\supp u\subset B_{R_0}$ and $\tau > 8MR_0$, we have
\begin{align*}
   \norm{u}_{X^\half_{-\tau\theta}}=&\norm{Q_*u}_{X^\half_{-\tau e_n}}\leq C R_0\norm{e^{\varphi_{\tau e_n}}\Delta(e^{-\varphi_{\tau e_n}}Q_* u)}_{X_{-\tau e_n}^{-\half}}=\\
   &C R_0\norm{Q_*[e^{\varphi_{\tau e_n}}\Delta(e^{-\varphi_{\tau e_n}} u)]}_{X_{-\tau e_n}^{-\half}}=C R_0\norm{e^{\varphi_{\tau \theta}}\Delta(e^{-\varphi_{\tau \theta}} u)}_{X_{-\tau \theta}^{-\half}},
\end{align*}
where we have used Theorem \ref{thm:cgo} and the fact that rotations commute both with the laplacian and the Fourier transform.
On the other hand, if $\zeta=\tau\theta+i\Ii$ and we denote by $\mathcal{N}_\Ii$ the translation operator by $\Ii\in\Rd$, it's again easy to check that $p_{-\zeta}=\mathcal{N}_\Ii p_{-\tau\theta}$. Therefore, keeping in mind that $\mathcal{N}_{-\Ii}\widehat{u}(\xi)=\widehat{e^{-i\Ii\cdot}u}(\xi)$, we obtain
\begin{align}\label{eq:carleman1}
    \begin{split}
        \norm{u}_{X^\half_{-\zeta}}&=\norm{e^{-i\Ii\centerdot}u}_{X^\half_{-\tau\theta}}\leq  C R_0\norm{e^{\varphi_{\tau \theta}}\Delta(e^{-\varphi_{\tau \theta}}e^{-i\Ii\centerdot} u)}_{X_{-\tau \theta}^{-\half}} \\
        &=C R_0\norm{e^{-i\Ii\centerdot}e^{\varphi_{\zeta}}\Delta(e^{-\varphi_{\zeta}} u)}_{X_{-\tau \theta}^{-\half}}=C R_0\norm{e^{\varphi_{\zeta}}\Delta(e^{-\varphi_{\zeta}} u)}_{X_{-\zeta}^{-\half}}
    \end{split}
\end{align}
Now we can perturb this inequality with the term $\lambda-V$ in the operator. The following inequalities will be of interest during the whole argument. For any $\zeta\in\mathbb{C}^d$ such that $\zeta\cdot\zeta\leq0$ and $u$ compactly supported, we have the following 
\begin{align}
    &\norm{u}_{L^2}\leq M^{-1/4}|\Re(\zeta)|^{-\half}\norm{u}_{X^\half_\zeta},\label{eq:desx1}\\
    &\norm{u}_{L^{p_d}}\lesssim M^{1/4}\norm{u}_{X^\half_\zeta}.\label{eq:desx2}
\end{align}
Inequality \eqref{eq:desx2} is a direct consequence of Haberman's embedding \cite{habermanuniqueness2015} 
\begin{equation}\label{eq:haber}
    \norm{u}_{L^{p_d}}\lesssim \norm{u}_{\dot X^\half_\zeta},
\end{equation}
where the space $\dot X^s_\zeta$ is defined by the norm $\norm{u}_{\dot X_\zeta^s}=\norm{|p_\zeta|^{s}\,\widehat{u}}_{L^2}.$
We are going to quantify $(\lambda-V)\,u$ in $X^{-\half}_\zeta$ by duality. Remember that in our case $|\Re(\zeta)|=\tau$. Start by estimating $\lambda u$ with $\mathscr{S}(\Rd)$. Let $v\in\mathscr{S}(\Rd)$, then by Cauchy-Schwartz inequality,
\begin{equation}\label{eq:pert1}
    \lvert\braket{\lambda\,u,v}\rvert\leq \lambda \norm{u}_{L^2}\norm{v}_{L^2}\leq \lambda M^{-\half}\tau^{-1}\norm{u}_{X^{\half}_\zeta}\norm{v}_{X^{\half}_\zeta}
\end{equation} 
Next, to estimate the term $V^0 u$, we can split the potential as in the proof of lemma \ref{lemma:int}. Indeed, consider the set $F=\{x\in\Rd\,:\,|V^0(x)|>N\}$, and define $E=\Rd\setminus F$. Then, $V^0=\mathbbm{1}_E V_0+\mathbbm{1}_F V_0$, with $\norm{\mathbbm{1}_E V_0}_{L^\infty}=N$, and therefore by Cauchy-Schwartz and Hölder inequalities, as well as \eqref{eq:desx1} and \eqref{eq:desx2},
\begin{align}\label{eq:pert2}
    \begin{split}
        \lvert\braket{V^0\,u,v}\rvert&\leq N\norm{u}_{L^2}\norm{v}_{L^2}+\norm{\mathbbm{1}_F V_0}_{L^{d/2}}\norm{u}_{L^{p_d}}\norm{v}_{L^{p_d}}.\\
    &\leq (NM^{-\half}\tau^{-1}+M^\half\norm{\mathbbm{1}_F V_0}_{L^{d/2}})\,\norm{u}_{X^{\half}_\zeta}\norm{v}_{X^{\half}_\zeta}.
    \end{split}
\end{align}
For the term $\alpha\,\dd\sigma$, we need to use the Besov space version of Theorem 14.1.1 in \cite{hormanderlinear1963}, which gives us a trace boundedness $\norm{u}_{L^2(\Gamma)}\leq \norm{u}_{\dot B^\half_{2,1}}$. Remember that the Besov spaces $\dot B_{p,q}^s$ are given by the following norms, using Littlewood-Paley projectors as defined in \eqref{eq:LP}:
\[ \norm{u}_{\dot B_{p,q}^s}=\left(\sum_{k\in\mathbb{Z}}2^{kqs}\norm{P_k u}^q_{L^p}\right)^{1/q}. \] Indeed, by Cauchy-Schwartz and Hölder inequalities,
\begin{align*}
    \begin{split}
        \lvert\braket{\alpha\dd\sigma\,u,v}\rvert&\leq \norm{\alpha}_{L^\infty(\Gamma)}\norm{u}_{L^2(\Gamma)}\norm{v}_{L^2(\Gamma)}\leq \norm{\alpha}_{L^\infty(\Gamma)} \norm{u}_{\dot B^\half_{2,1}}\norm{v}_{\dot B^\half_{2,1}}
    \end{split}
\end{align*}
Now, estimate separately high and low frequencies. Let $k_\tau\in\mathbb{Z}$ be such that $2^{k_\tau-1}<\tau\leq 2^{k_\tau}$. Then, if $k>k_\tau+1$, we have that $2^{\hlf{k}}|\widehat{P_ku(\xi)}|\sim 2^{-\hlf{k}}|p_\zeta(\xi)|^\half|\widehat{P_ku(\xi)}|$, so that for the high frequencies we have that, by Plancherel's identity,
\begin{align*}
    \sum_{k>k_\tau+1}2^{\hlf{k}}\norm{P_ku}_{L^2}\sim\sum_{k>k_\tau+1}2^{-\hlf{k}}\norm{|p_\zeta|^\half\widehat{P_ku}}_{L^2}\leq \tau^{-\half}M^{1/4}\norm{u}_{X_\zeta^\half},
\end{align*}
while for the low frequencies 
\[ \sum_{k\leq k_\tau+1}2^{\hlf{k}}\norm{P_ku}_{L^2}\lesssim \tau^\half\norm{u}_{L^2}\leq M^{-1/4}\norm{u}_{X^\half_\zeta}. \] 
Therefore, combining the previous inequalities, we have that there exist a constant $C'>0$ such that
\begin{equation}\label{eq:pert3}
    \lvert\braket{\alpha\dd\sigma\,u,v}\rvert\leq C'\norm{\alpha}_{L^\infty(\Gamma)}(M^{-\half}+\tau^{-\half}+\tau^{-1}M^\half)\norm{u}_{X^{\half}_\zeta}\norm{v}_{X^{\half}_\zeta}.
\end{equation}
Finally, for $\gamma^s$, use proposition \ref{prop:vf} to obtain
\begin{align*}
    \lvert\braket{\gamma^s u,v}\rvert\leq C''\norm{g}_{L^\infty}\left(\norm{D^su}_{L^2}\norm{v}_{L^2}  + \norm{u}_{L^2}\norm{D^sv}_{L^2}\right).
\end{align*}
Split again in high and low frequencies:
\begin{align*}
    \norm{D^su}_{L^2}\leq&\int_{\Rd}|\xi|^{2s}\,|\widehat{u}(\xi)|^2\dd\xi=\int_{|\xi|<\tau}|\xi|^{2s}\,|\widehat{u}(\xi)|^2\dd\xi+\int_{|\xi|\geq\tau}|\xi|^{2s}\,|\widehat{u}(\xi)|^2\dd\xi.
\end{align*}
On the one hand, for the low frequencies, clearly
\[ \int_{|\xi|<\tau}|\xi|^{2s}\,|\widehat{u}(\xi)|^2\dd\xi\leq \,\tau^{2s}\,\norm{u}_{L^2}^2\leq\,\tau^{2s-1}M^{-\half}\,\norm{u}^2_{X_\zeta^\half}, \]
while for the high frequencies, having in mind that we assumed $s<1$,
\begin{align*}
    \int_{|\xi|\geq\tau}|\xi|^{2s}\,|\widehat{u}(\xi)|^2\dd\xi\leq& \int_{|\xi|\geq\tau}|\xi|^{2(s-1)}\,|p_\zeta(\xi)|\,\widehat{u}(\xi)\dd\xi\leq\tau^{2(s-1)} \,\norm{|p_\zeta|^\half\,\widehat{u}}^2_{L^2}\\\
    \leq& \,\tau^{2(s-1)}M^\half\, \norm{u}^2_{X_\zeta^\half}.
\end{align*}
Therefore, we have that
\[ \norm{D^su}_{L^2}\leq \left(\tau^{s-\half}M^{-1/4}+\tau^{s-1}M^{1/4}\right)\norm{u}_{X_\zeta^\half}, \]
which yields
\begin{equation}\label{eq:pert4}
    \lvert\braket{\gamma^s u,v}\rvert \leq C''\norm{g}_{L^\infty}\left(\tau^{s-1}M^{-\half}+\tau^{s-3/2}\right)\norm{u}_{X^{\half}_\zeta}\norm{v}_{X^{\half}_\zeta}.
\end{equation}
Now, sum inequalities \eqref{eq:pert1}, \eqref{eq:pert2}, \eqref{eq:pert3} and \eqref{eq:pert4}, and choose $M>CR_0^2$ so that $CR_0C'M^{-\half}\norm{\alpha}_{L^\infty(\Gamma)}\leq1/4$, then choose $N$ such that $CR_0M^{\half}\norm{\mathbbm{1}_FV^0}_{L^{d/2}}\leq1/4,$ and finally choose $\tau>8MR_0$ so that
\begin{align*}
    CR_0\Bigl[\left(\lambda+N\right)M^{-\half}\tau^{-1} &+ C'\norm{\alpha}_{L^\infty(\Gamma)}\left(\tau^{-\half}+\tau^{-1}M^\half\right) \\
    & +C''\norm{g}_{L^\infty}\left(\tau^{s-1}M^{-\half}+\tau^{-3/2}\right) \Bigr] <1/4.
\end{align*}
With this previous discussion we can prove what is summarized in the following lemma:
\carlemanen
Next, note that
\[ \Delta(e^{-\zeta\cdot x}u)=e^{-\zeta\cdot u}(\Delta-2\zeta\cdot\nabla+\zeta\cdot\zeta)\,u, \]
and thus lemma \ref{lem:carlemanest} yields
\[ \norm{u}_{X^\half_{-\zeta}}\leq C R_0\norm{e^{M\frac{(\centerdot\cdot\theta)^2}{2}}(\Delta-2\zeta\cdot\nabla+\zeta\cdot\zeta+\lambda-V)(e^{-M\frac{(\centerdot\cdot\theta)^2}{2}} u)}_{X_{-\zeta}^{-\half}}. \]
Now we procceed to remove the remaining exponential factors. Take $u=e^{M\frac{(\centerdot\cdot\theta)^2}{2}}v$ with $v\in\mathscr{S}(\Rd)$ supported in $B_{R_0}$, then
\begin{equation}\label{eq:jejejo}
    \norm{e^{M\frac{(\centerdot\cdot\theta)^2}{2}}v}_{X^\half_{-\zeta}}\leq C R_0\norm{e^{M\frac{(\centerdot\cdot\theta)^2}{2}}(\Delta-2\zeta\cdot\nabla+\zeta\cdot\zeta+\lambda-V)v}_{X_{-\zeta}^{-\half}}.
\end{equation}
Additionaly, if we prove that
\begin{equation}\label{eq:jeje1}
    \norm{e^{M\frac{(\centerdot\cdot\theta)^2}{2}}\chi w}_{X^\half_{\zeta}}\lesssim \norm{w}_{X^\half_{\zeta}},
\end{equation}
it will follow by duality that
\[ \norm{e^{M\frac{(\centerdot\cdot\theta)^2}{2}}v}_{X^\half_{-\zeta}}\lesssim\norm{(\Delta-2\zeta\cdot\nabla+\zeta\cdot\zeta+\lambda-V)v}_{X_{-\zeta}^{-\half}}, \]
while if 
\begin{equation}\label{eq:jeje2}
    \norm{e^{-M\frac{(\centerdot\cdot\theta)^2}{2}} w}_{X^\half_{\zeta}}\lesssim \norm{w}_{X^\half_{\zeta}},
\end{equation}
then again by duality \[ \norm{v}_{X_{-\zeta}^\half}\lesssim\norm{e^{M\frac{(\centerdot\cdot\theta)^2}{2}}v}_{X_{-\zeta}^\half}. \]
Above $\chi(x\cdot\theta)=\chi_0(x\cdot\theta/R)$ with $\chi_0$ a cutoff $C_0^\infty(\mathbb{R};[0,1])$ function such that $\chi_0(t)=1$ for $|t|\leq2$ and $\chi_0(t)=0$ for $|t|>4$.
Putting inequalities \eqref{eq:jejejo}, \eqref{eq:jeje1} and \eqref{eq:jeje2} we will obtain
\[ \norm{v}_{X_{-\zeta}^\half}\leq C' R_0\norm{(\Delta-2\zeta\cdot\nabla+\zeta\cdot\zeta+\lambda-V)v}_{X_{-\zeta}^{-\half}}, \]
with a new constant $C'>0$. To indeed prove \eqref{eq:jeje1} and \eqref{eq:jeje2} we draw next lemma from \cite{caroglobal2016}:
\begin{lemma}[\cite{caroglobal2016}, lemma 2.2]
    Let $f\in\mathscr{S}(\mathbb{R})$ be a function of the $x_n$ variable. If $u\in\mathscr{S}(\Rd)$ and $\tau>M>1,$ then
    \[ \norm{fu}_{X^\half_{\tau e_n}}\lesssim \norm{p\widehat{f}}_{L^1(\mathbb{R})}\norm{u}_{X^\half_{\tau e_n}}, \]
    where $p(\xi)=(M^{-1}|\xi|+1)^2$.
\end{lemma}
As before, we can transform the inequality to make it work for an arbitrary $\zeta$ as we did above, and for a $f$ that depends on the variable $x\cdot\theta$, which can be represented as $Q^T_*f$, with $\theta=Q e_n$. On the one hand,
\[ \norm{Q^T_*fu}_{X^\half_{\tau \theta}}=\norm{fQ_*u}_{X^\half_{\tau e_n}}\lesssim\norm{p\widehat{f}}_{L^1(\mathbb{R})}\norm{Q_*u}_{X^\half_{\tau e_n}}=\norm{p\widehat{f}}_{L^1(\mathbb{R})}\norm{u}_{X^\half_{\tau\theta}}, \]
while if $\zeta=\tau\theta+i\Ii$, then
\[ \norm{fu}_{X^\half_{\zeta}}=\norm{e^{i\Ii\centerdot}fu}_{X^\half_{\tau \theta}}\lesssim\norm{p\widehat{f}}_{L^1(\mathbb{R})}\norm{e^{i\Ii\centerdot}u}_{X^\half_{\tau\theta}}=\norm{p\widehat{f}}_{L^1(\mathbb{R})}\norm{u}_{X^\half_{\zeta}}. \]
We summarize all this discussion in the following lemma:
\begin{lemma}\label{lem:apriori}
   Take $R_0>0$ such that $\supp V\subset B_{R_0}$. There exists $C>0$ and $\tau_0=\tau_0(R_0,V,\lambda)$ such that 
\[ \norm{u}_{X_{-\zeta}^\half}\leq C R_0\norm{(\Delta-2\zeta\cdot\nabla+\zeta\cdot\zeta+\lambda-V) u}_{X_{-\zeta}^{-\half}} \]
for all for $u\in\mathscr{S}(\Rd)$ with $\supp u\subset B_{R_0}$ and $\tau > \tau_0$.
\end{lemma}
To prove the existence of CGO solutions we should introduce a couple of natural spaces. Let $\D$ be a bounded open domain of class $\C^2$ such that $\supp V_j\subset \D$, $j=1,2$, and, for $\zeta\in\mathbb{C}^d$ and $s>0$ define
\[ X_\zeta^s(\D)=\{\restr{u}{\D}\,:\,u\in X_\zeta^s\}, \]
endowed with the norm
\[ \norm{u}_{X_\zeta^s(\D)}=\inf\{\norm{v}_{X_\zeta^s}\,:\,\restr{v}{\D}=u\}, \]
and 
\[ X_{\zeta,c}^s(\D)=\{u\in X^s_\zeta\,:\,\supp\,u\subset \overline{\D}\}, \]
endowed as well with the norm $\norm{\centerdot}_{X_\zeta^s(\D)}$. The space $X_{-\zeta,c}^{-s}(\D)$ is defined as the dual space of $X_\zeta^s(\D)$, and it is easy to check that it can be identified with the space of distributions in $X_{-\zeta}^{-s}$ whose support lays inside $\D$. In this context, the a priori estimate in lemma \ref{lem:apriori} works in the space $X_{-\zeta,c}^{-\half}(\D)$, so that the solutions can be constructed in $X_\zeta^\half(\D)$, as stated in the following proposition.
\begin{prop}\label{prop:cgo}
    Consider $d\geq 3$, let $\D$ be a bounded open domain of class $\C^2$ amd let $R_0>0$ such that $\supp V\subset\D\subset B_{R_0}$. There exist a constant $\tau_0=\tau_0(R_0,V,\lambda)$ such that, for every $\tau\geq\tau_0$, and every $\zeta=\Re(\zeta)+i\Im(\zeta)\in\mathbb{C}^d$ such that $|\Re(\zeta)|=\tau$, $|\Im(\zeta)|=(\tau^2+\lambda)^\half$ and $\Re(\zeta)\cdot\Im(\zeta)=0$, there exist $w_\zeta\in X_\zeta^\half(\D)$ so that $v_\zeta=e^{\zeta\cdot x}(1+w_\zeta)$ is a solution to the equation $(\Delta+\lambda-V)v_\zeta=0$ in $\D$ and
    \begin{equation}\label{eq:descgo}
        \norm{w_\zeta}_{X_\zeta^\half(\D)}\lesssim \norm{V}_{X_\zeta^{-\half}}.
    \end{equation}
\end{prop}
        \subsection{Proof of theorem \ref{thm:1.2}}\label{sect:inversethm}
        With proposition \ref{prop:cgo}, we can construct the kind of special solutions that we are looking for, with $\zeta_j\cdot\zeta_j=-\lambda$. Besides, we need $\zeta_j$ to satisfy that $\zeta_1+\zeta_2=-i\kappa$ for an arbitrary given $\kappa\in\mathbb{R}^d$. We can explicitly construct these $\zeta_j$ by choosing $\eta,\theta\in\mathbb{S}^{d-1}$ such that $\eta\cdot\theta=\eta\cdot\kappa=\theta\cdot\kappa=0$. Now, for $\tau\geq\frac{|\kappa|^2}{4}-\lambda$ we can set
\begin{align}\label{eq:zeta}
    \begin{split}
    \zeta_1&=\tau\theta+i\left[-\frac{\kappa}{2}+(\tau^2+\lambda-\frac{|\kappa|^2}{4})^\half\eta\right],\\
    \zeta_2&=-\tau\theta+i\left[-\frac{\kappa}{2}-(\tau^2+\lambda-\frac{|\kappa|^2}{4})^\half\eta\right],
    \end{split}
\end{align}
which satisfy both $\zeta_j\cdot\zeta_j=-\lambda$ and $\zeta_1+\zeta_2=-i\kappa$. Then, if we take $\tau\geq \max\{\tau_0,(\frac{|\kappa|^2}{4}-\lambda)^\half\}$, these $\zeta_j$ satisfy the conditions of proposition \ref{prop:cgo}. Let then $V_1$ and $V_2$ be two potentials of the form \eqref{V}, and let $v_j$ be the CGO solutions corresponding to $\zeta_j$ and $V_j$, $j=1,2$. If we consider any extension of $w_j\in X_{\zeta_j}^\half(\D)$ to $X_{\zeta_j}^\half$, this extension will be in $H^1(\Rd)$. Then, $w_j$ belongs to $H^1(\D)$ and so does $v_j$. Therefore, we can apply proposition \ref{proportho} and, plugging $v_1$ and $v_2$ in the orthogonality relation \[ \braket{(V_1-V_2)v_1,v_2}=0, \]
we get
\begin{align}\label{eq:ales}
    \begin{split}
        \braket{V_1-V_2,e^{-{i\kappa\cdot x}}}=&-\braket{V_1-V_2,e^{-i\kappa\cdot x}w_1}-\braket{V_1-V_2,e^{-i\kappa\cdot x}w_2}\\
        &-\braket{(V_1-V_2)w_1,e^{-i\kappa\cdot x}w_2}. 
    \end{split}
\end{align}
We would like this terms to vanish. For the first two terms on the right hand side we have that, by duality,
\begin{align}\label{eq:ales1}
    \begin{split}
    \lvert\braket{V_1-V_2,e^{-i\kappa\cdot x}w_j}\rvert&\leq\norm{V_1-V_2}_{X_{\zeta_j,c}^{-\half}(\D)}\norm{e^{-i\kappa\cdot x}w_j}_{X_{\zeta_j}^\half(\D)}\\
    &\lesssim (1+|\kappa|)\, \norm{V_1-V_2}_{X_{\zeta_j}^{-\half}}\norm{V_j}_{X_{\zeta_j}^{-\half}},
    \end{split}
\end{align}
where in the last inequality we have used that $\supp(V_1-V_2)\subset \D$, inequality \eqref{eq:descgo} and the following:
\begin{equation}\label{eq:kappa}
    \norm{e^{-i\kappa\cdot x}w_j}_{X_{\zeta_j}^\half(\D)}\lesssim (1+|\kappa|)\,\norm{w_j}_{X_{\zeta_j}^\half(\D)}.
\end{equation}
On the other hand, the third term can be bounded again by duality as
\begin{align*}
    \lvert\braket{(V_1-V_2)w_1,e^{-i\kappa\cdot x}w_2}\rvert&\leq \norm{(V_1-V_2)w_1}_{X_{\zeta_2,c}^{-\half}(\D)}\norm{e^{-i\kappa\cdot x}w_2}_{X_{\zeta_2}^\half(\D)}\\
    &\lesssim (1+|\kappa|)\, \norm{(V_1-V_2)w_1}_{X_{\zeta_2}^{-\half}}\norm{V_2}_{X_{\zeta_2}^{-\half}}
\end{align*}
where we have used again \eqref{eq:descgo}, \eqref{eq:kappa} and the fact that $\supp(V_1-V_2)w_1\subset \D$. Now, to keep estimating this term, we need boundedness of the operator multiplication by $V_1-V_2$ from $X_{\zeta_1}^\half(\D)$ to $X_{\zeta_2}^{-\half}$. In fact, let $V$ be a potential of the form \eqref{V} and $w\in X_{\zeta_1}^\half$. Proceeding just as in section \ref{sect:carleman}, we obtain that, for $\phi\in X_{\zeta_2}^\half$,
\[ \lvert\braket{V\,w,\phi}\rvert\lesssim (\norm{V^0}_{L^{\hlf{d}}}+\norm{\alpha}_{L^\infty(\Gamma)}+\norm{g}_{L^\infty})\,\norm{u}_{X_{\zeta_1}^\half}\norm{\phi}_{X_{\zeta_2}^\half}, \]
where the implicit constant depends only on $V$ and $d$, and $u\in X^\half_\zeta$ is an arbitrary extension of $w$ to $\Rd$. Taking the infimum over the norm of $u$ gives us the desired boundedness, i.e. \[ \norm{(V_1-V_2)w_1}_{X_{\zeta_2}^{-\half}}\lesssim\norm{w_1}_{X_{\zeta_1}^\half(\D)}, \] to estimate the last summand in \eqref{eq:ales} as
\begin{align}\label{eq:ales2}
    \begin{split}
    \lvert\braket{(V_1-V_2)w_1,e^{-i\kappa\cdot x}w_2}\rvert&\lesssim (1+|\kappa|)\, \norm{w_1}_{X_{\zeta_1}^\half(\D)}\norm{V_2}_{X_{\zeta_2}^{-\half}}\\
    &\lesssim (1+|\kappa|)\, \norm{V_1}_{X_{\zeta_1}^{-\half}}\norm{V_2}_{X_{\zeta_2}^{-\half}},
    \end{split}
\end{align}
where again we have used \eqref{eq:descgo}. Therefore, adding inequalities \eqref{eq:ales1} and \eqref{eq:ales2} gives
\begin{equation}\label{eq:invfou}
    \lvert\braket{V_1-V_2,e^{-{i\kappa\cdot x}}}\rvert\lesssim(1+|\kappa|)\,\sum_{j,k=1}^2\norm{V_j}_{X_{\zeta_k}^{-\half}}\,\sum_{l,m=1}^2\norm{V_l}_{X_{\zeta_m}^{-\half}}
\end{equation}
We now want to show that the right-hand side of \eqref{eq:invfou} goes to zero in some sense as $\tau$ grows. This can be done in average, as showed by Haberman and Tataru in \cite{habermanuniqueness2013}. We state it in the following lemma, almost identical to lemma 3.5 in \cite{caroscattering2020}, with the only addition of the new component of the potential $\gamma^s$. We repeat the proof here just for the sake of completeness.
\begin{lemma}\label{lem:habibi}
    Let $V$ be a potential of the form \eqref{V} and fix $\nu\in\S^{d-1}$ and $r>0$. For $\zeta\in\mathbb{C}^d$ of the form \eqref{eq:zeta}, denote $\zeta=\zeta(\tau,T)$ with $\kappa=rT\nu$ for some $T\in SO(d)$. Then for every $s\in (1/2,1)$ and $M>1$, it holds that
    \[ \frac{1}{M}\int_M^{2M}\int_{SO(d)}\norm{V}_{X^{-1/2}_{\zeta(\tau,T)}}^2\dd\mu(T)\dd\tau\lesssim M^{-\omega}\norm{V_0}^2_{L^{d/2}}+M^{-2(1-s)}\left(\norm{\alpha}^2_{L^\infty(\Gamma)}+\norm{g}_{L^\infty}^2\right), \]
    where     $$\omega=\begin{cases}1/2 & d=3\\
         1/4 & d\geq 4,\end{cases}$$ the implicit constant depends on $V$ and $d$, and $\mu$ denotes the Haar measure on $SO(d)$.
\end{lemma}
\begin{proof}
    First, for the critically singular part, if $d\geq 4$,
    \[ \norm{V^0}_{X^{-\half}_{\zeta(\tau,T)}}\leq\tau^{-\half}\norm{V^0}_{L^2}\lesssim\tau^{-\half}\norm{V^0}_{L^{d/2}}, \]
    since $V^0$ is compactly supported and $d/2\geq 2$ for $d\geq4$. In the case $d=3$, by the dual inequality to Haberman's embedding (see corollaries 4.8 and 4.22 in \cite{caroscattering2020}),
    \[ \norm{V^0}_{X^{-\half}_{\zeta(\tau,T)}}\leq\tau^{-d(1/p'_d-1/q'_d)}\norm{V^0}_{L^{q'_d}}\lesssim\tau^{-1/4}\norm{V^0}_{L^{d/2}}, \]
    since $d/2\geq q'_d$ for $d=3$. For the remaining components, use lemma 5.2 in \cite{habermanuniqueness2015}, which states that for $f\in \dot H^{-1}$, it holds that
    \[ \frac{1}{M}\int_M^{2M}\int_{SO(d)}\norm{f}_{X^{-1/2}_{\zeta(\tau,T)}}^2\dd\mu(T)\dd\tau\lesssim M^{-1}\norm{Lf}^2_{\dot H^{-\half}}+\norm{Hf}^2_{\dot H^{-1}}, \]
    where $\widehat{Lf}=\mathbbm{1}_{|\xi|<2M}\widehat{f}$ and $Hf=f-Lf$ stand the low and high frequencies of $f$, respectively. Now, for every $1/2<s<1$, we obtain
    \[ \frac{1}{M}\int_M^{2M}\int_{SO(d)}\norm{f}_{X^{-1/2}_{\zeta(\tau,T)}}^2\dd\mu(T)\dd\tau\lesssim M^{-2(1-s)}\norm{f}^2_{\dot H^{-s}}. \]
    On the one hand, since $\supp (\alpha\dd\sigma)\subset \D$,
    \[ \norm{\alpha\dd\sigma}^2_{\dot H^{-s}}\lesssim\norm{\alpha\dd\sigma}^2_{H^{-s}}\lesssim \norm{\alpha}_{L^2(\Gamma)}\lesssim\norm{\alpha}_{L^\infty(\Gamma)}, \]
    where we have used the dual inequality of the usual trace theorem for Sobolev spaces, as well as the fact that $\Gamma$ has finite measure. On the other hand, since $\gamma^s=\chi^2D^s\,g$ with $\supp \chi\subset \D$, it follows that for $u\in \mathcal{S}(\Rd)$,
    \[ \lvert\braket{\gamma^s,u}\rvert\leq \norm{g}_{L^\infty(\Rd)}\norm{D^s(\chi^2 u)}_{L^{2}}\lesssim \norm{g}_{L^\infty(\Rd)}\norm{u}_{\dot H^s}, \]
    where the last inequality comes by proceededing as in \eqref{eq:chi}. Therefore,
    \[ \norm{\gamma^s}_{\dot H^{-s}}\lesssim \norm{g}_{L^\infty(\Rd)}, \]
    which ends the proof of the lemma.  
\end{proof}
To end the proof of Theorem \ref{thm:1.2}, we use an argument that Haberman \cite{habermanuniqueness2015} attributes to Nguyen and Spirn \cite{nguyen}. Indeed, if we fix $r>0$ and $\nu\in\mathbb{S}^{d-1}$ and denote $\zeta=\zeta(\tau,T)$ as in lemma \ref{lem:habibi} above, we have that
\[ \lim_{M\to\infty}\frac{1}{M}\int_M^{2M}\int_{SO(d)}\norm{V}_{X^{-1/2}_{\zeta(\tau,T)}}^2\dd\mu(T)\dd\tau=0, \]
and, for any $\varepsilon>0$, by simply restricting to $B_\varepsilon=\{T\in SO(d):\norm{T-I}\leq\varepsilon\}$ with $I$ the identity map, we get that
\[ \lim_{M\to\infty}\frac{1}{M\mu(B_\varepsilon)}\int_M^{2M}\int_{B_\epsilon}\norm{V}_{X^{-1/2}_{\zeta(\tau,T)}}^2\dd\mu(T)\dd\tau=0. \]
Now, take a sequence $M=M_n$ such that $M_n\to\infty$ as $n\to\infty$. Since the quantity in the limit above is an average, we may choose sequences $\tau_{\varepsilon,n}$, $T_{\varepsilon,n}$ and $\delta_{\varepsilon,n}>0$ such that
\[ \norm{V}_{X^{-1/2}_{\zeta(\tau_{\varepsilon,n},T_{\varepsilon,n})}}\leq\delta_{\varepsilon,n}, \]
and such that $\delta_{\varepsilon,n}\to0$ as $n\to\infty$. Therefore, going back to \eqref{eq:invfou}, we obtain
\[ |\widehat{V_1-V_2}(\kappa_{\varepsilon,n})|=\lvert\braket{V_1-V_2,e^{-{i\kappa_{\varepsilon,n}\cdot x}}}\rvert\lesssim \delta_{\varepsilon,n}^2, \]
where $\kappa_{\varepsilon,n}=rT_{\varepsilon,n}\nu$. Since $B_\varepsilon$ is compact, there exists a subsequence $T_{\varepsilon,n_m}$ converging to some $T_\varepsilon$. Thus,
\[ \lim_{m\to\infty}|\widehat{V_1-V_2}(\kappa_{\varepsilon,n_m})|\lesssim \lim_{m\to\infty} \delta_{\varepsilon,n_m}^2=0, \]
and, since $V_j$ are compactly supported, $\widehat{V_1-V_2}$ is continuous, which means that
\[ \widehat{V_1-V_2}(rT_\varepsilon \nu)=0. \]
Finally, since necessarily $T_\varepsilon\to I$ as $\varepsilon\to0$, we conclude that
\[ \widehat{V_1-V_2}(r\nu)=0 \] for any $r>0$ and $\nu\in\S^{d-1}$ and, by Fourier inversion we obtain that $V_1=V_2$, which ends the proof of Theorem \ref{thm:1.2}.

\appendix

\section{Solution of the Neumann problem}\label{ap:neumann}

Along this section, we aim to prove the following result, which is used in the proof of the Runge approximation, proposition \ref{prop:runge}. Throughout this appendix, consider a bounded open domain $\Omega$ of class $\C^2$ such that $\supp V\subset \Omega$.
\begin{theorem}\label{thm:neu}
    Suppose $\lambda>0$ is \textbf{not} a Neumann eigenvalue for the operator $-\Delta+V$ in $\Omega$. Let $f\in L^2(\Omega)$ be such that $\supp f\subset \Omega$. Then, there exist $u\in H^1(\Omega)$ solving the problem
    \begin{equation}\label{eq:neumann1}
        \begin{cases}
                \left(\Delta +\lambda- V\right)u= f & \textrm{in }\Omega,\\
                \partial_{\nu}u=0 & \textrm{on }\partial \Omega.
            \end{cases}
    \end{equation}
\end{theorem}
Remember that we say that $\lambda$ is a Neumann eigenvalue for $-\Delta+V$ in $\Omega$ if there exists $\phi$ not identically zero solving the homogeneuos Neumann problem.
\begin{equation}\label{eq:homogeneousneumann}
    \begin{cases}
            \left(\Delta +\lambda- V\right)\phi= 0 & \textrm{in }\Omega,\\
            \partial_{\nu}\phi=0 & \textrm{on }\partial \Omega.
        \end{cases}
\end{equation}
These eigenvalues in fact form a countable set. To see this, define the unbounded operator $\left(T_N,D(T_N)\right)$ over $L^2(\Omega)$ as $T_Nu=\left(-\Delta+V\right)u$, with domain 
\[ D(T_N)=\{u\in L^2(\Omega)\,:\, \left(-\Delta+V\right)u\in L^2(\Omega),\; \exists\, \partial_\nu u\text{ on } \partial\Omega \text{ and }\partial_\nu \restr{u}{\partial \Omega}=0\}.\]
Observe now that a Neumann eigenvalue for $-\Delta+V$ on $\Omega$ will be an eigenvalue for $\left(T_N,D(T_N)\right)$. The domain $D(T_N)$ is a separable Hilbert space, since it is a subspace of $L^2(\Omega)$, and $T_N$ is symmetric over $D(T_N)$, which ensures that the set of its eigenvalues must be countable. \\
Indeed, suppose that there is an uncountable set of such eigenvalues. Let $\lambda$ and $\mu$ be any two distinct eigenvalues, and $u$ and $v$ be corresponding distinct eigenfunctions. Then,
\[
  \lambda\braket{u,v}=\braket{T_N u, v}=\braket{u,T_N v}= \mu\braket{u,v},
  \]
and thus $u\perp v$, which contradicts the fact that $D(T_N)$ is separable.\\
We will now prove Theorem \ref{thm:neu} using the method of layer potentials. Most of the arguments that we will follow come from \cite{coltoninverse2013}, \cite{follandintroduction1995} and \cite{coltonintegral}. The first step in the argument is to realise that, to solve \eqref{eq:neumann1}, we only need to be able to solve the following problem, for $g\in \mathcal{C}^\infty(\partial \Omega)$:
\begin{equation}\label{eq:neumann2}
    \begin{cases}
            \left(\Delta +\lambda- V\right)v= 0 & \textrm{in }\Omega,\\
            \partial_{\nu}v=g & \textrm{on }\partial \Omega.
        \end{cases}
\end{equation}
This can be done by taking a function $w$ that satisties the equation \[ (\Delta+\lambda-V)\,w=f. \] This function can be constructed in $X_\lambda^*$, for instance, by using the techniques in section \ref{sect:direct} after observing that, if $f\in L^2(\Omega)$, then its extension by $0$ to $\Rd$ is in $ X_\lambda$.\\
Also, since $\supp f\subset \Omega$, we have by Theorem 11.1.1 in \cite{hormanderlinear1963} that $w$ is smooth near $\partial \Omega$. Therefore, we can define its outward normal derivative in $\partial \Omega$, which will belong to $\mathcal{C}^\infty(\partial \Omega)$.
 Now, if we can solve problem \eqref{eq:neumann2}, set $g=\partial_\nu w$, and denote by $v$ the solution to this problem. Then, it is easy to check that $u=\restr{w}{\Omega}-v$ solves the problem \eqref{eq:neumann1}.\\
  Note that $\restr{w}{\Omega}\in H^1(\Omega)$ by proposition \ref{prop:restriction}, so we will only need to prove that $v$ belongs to $H^1(\Omega)$ to conclude that $u$ belongs too.\\
In this case it will be useful to think of the fundamental solution for the operator $\Delta+\lambda$ in $\Rd$, denoted by $\Phi_\lambda$, not in the distributional sense as in \eqref{eq:fundsol}, but as a Hankel function. In fact, $\Phi_\lambda$ will take the form
\begin{equation}
	\Phi_\lambda(x)=\frac{i}{4}\left(\frac{\lambda^\half}{2\pi|x|}\right)^{d/2-1}H_{d/2-1}^{(1)}\left(\lambda^\half |x|\right),
\end{equation}
where $H_\nu^{(1)}$ denotes the Hankel function of the first kind (or Bessel function of the third kind). If we define $u_{in}(x,y)=\Phi_\lambda(x-y)$ and recall the limiting properties of the Hankel functions, it's relatively easy to check that
\begin{align}\label{eq:kernel}
    \begin{split}
	    u_{in}(x,y)&= F(x,y)|x-y|^{2-d},\\
	    \partial_{\nu_x}u_{in}(x,y)&=G(x,y)|x-y|^{2-d},\\
        \partial_{\nu_y}u_{in}(x,y)&=\partial_{\nu_x}u_{in}(y,x)=G(y,x)|x-y|^{2-d},
    \end{split}
\end{align}
with $F$ and $G$ being two bounded functions on $\partial \Omega\times\partial \Omega$. Then, $u_{in}$, $\partial_{\nu_x}u_{in}$ and $\partial_{\nu_y}u_{in}$ are, by definition, \textbf{weakly singular kernels} of order $d-2$ on $\partial \Omega\times\partial \Omega$. The following lemma is a combination of those in \cite{follandintroduction1995}, chapter 3B, and is an important piece to solve the problem \eqref{eq:neumann2}.
\begin{lemma}[\cite{follandintroduction1995}]\label{lem:kernel}
    If we denote by $T$ the integral operator over $\partial \Omega$ defined by a weakly singular kernel $K$ of order $\alpha$ on $\partial \Omega\times \partial \Omega$, with $0<\alpha<d-1$, as
    \[ \left(Tf\right)(x)=\int_{\partial \Omega}T(x,y)\,f(y)\dd S(y), \] then the following statements hold:
    \begin{enumerate}
        \item $T$ is compact on $L^2(\partial \Omega)$,
        \item $T$ transforms bounded functions into continuous functions, and
        \item if $f\in L^2(\partial \Omega)$ and $f+Tf\in \mathcal{C}(\partial \Omega)$, then $f\in\mathcal{C}(\partial \Omega)$.
    \end{enumerate}
\end{lemma}
Take now $u_{to}=u_{in}+u_{sc}$, where $u_{sc}$ is the scattering solution defined in \eqref{eq:helm} and constructed in section \ref{sect:direct}, and define for $f$ continuous on $\partial \Omega$ and $x\in \Rd\setminus \partial \Omega$ the \textbf{single layer potential} with moment $f$ as
\[ \left(\mathcal{S}f\right)(x)=\int_{\partial \Omega}u_{to}(x,y)\,f(y)\dd S(y), \]
and the \textbf{double layer potential} as 
\[ \left(\mathcal{D}f\right)(x)=\int_{\partial \Omega}\partial_{\nu_y}u_{to}(x,y)\,f(y)\dd S(y),. \]
Define further the operator $\mathcal{N}$, which is the adjoint of $\mathcal{D}$ over $\partial \Omega$, as
\[ \left(\mathcal{N}f\right)(x)=\int_{\partial \Omega}\partial_{\nu_x}u_{to}(x,y)\,f(y)\dd S(y),\quad x\in \partial \Omega, \]
that must be understood as an improper integral.
Now, we have the following lemmas, which are similar to classical results as in \cite{coltoninverse2013}, \cite{follandintroduction1995} and \cite{coltonintegral}  for the Helmholtz and Laplace equations. We denote by $u_+$ the trace on $\partial \Omega$ of $\restr{u}{\Rd\setminus \overline{\Omega}}$, and by $u_-$ the trace on $\partial \Omega$ of $\restr{u}{\Omega}$. As well, we denote by $\partial_\nu u_+$ and $\partial_\nu u_-$ the normal derivative of those, always with respect to the outward-pointing normal vector of $\partial \Omega$ (as seen from inside $\Omega$).
\begin{lemma}\label{lem:single}
    Consider $d\geq3$. Lef $f\in\mathcal{C}(\partial \Omega)$. Then, the single layer potential $u=\mathcal{S}f$  is continuous throughout $\Rd$, and we have the limiting values
    \begin{equation}\label{eq:bdsingle}
      \partial_\nu u_{\pm}(x)=\left(\mathcal{N}f\right)(x)\mp \frac{1}{2}f(x),\quad x\in\partial \Omega,
  \end{equation}
  where the integral exists as an improper integral.
  Consequently, we have the jump relation $\partial_\nu u_--\partial_\nu u_+ = f$ on $\partial \Omega$. Furthermore, $u$ it is a solution in $H^1_\text{loc}(\Rd)$ to $\helm u = 0$ in $\mathbb{R}^d\setminus \partial \Omega$ and fullfils SRC \eqref{eq:src}. 
\end{lemma}
\begin{proof}
    First, note that the single layer potential for the homogeneuos Helmholtz equation \[ v(x)=\int_{\partial \Omega}u_{in}(x,y)f(y)\dd S(y) \] can be extended to the boundary, is a solution in $H^1_\text{loc}(\Rd)$ to $\left(\Delta+\lambda\right) v= 0$ in $\mathbb{R}^d\setminus \partial \Omega$, fullfils SRC \eqref{eq:src} and has boundary values     
    \begin{equation*}
        \partial_\nu v_{\pm}(x)=\int_{\partial \Omega}\partial_{\nu_x}u_{in}(x,y)\,f(y)\dd S(y)\mp \frac{1}{2}f(x),\quad x\in\partial \Omega,
    \end{equation*} which is a classical result, see for example \cite{coltoninverse2013}.
    Now, define \[ w(x)=\int_{\partial \Omega}u_{sc}(x,y)f(y)\dd S(y). \] To see that $w$ is in $H^1_\text{loc}(\Rd)$, take $K\in\Rd$ compact, and observe that, by proposition \ref{prop:restriction} and Theorem \ref{thm:1.1},
    \[ \norm{w}_{H^1(K)}\lesssim \sup_{y\in\partial\Omega}\norm{u_{sc}({\centerdot\,,y})}_{H^1(K)}\lesssim\sup_{y\in\partial\Omega} \norm{u_{sc}({\centerdot\,,y})}_{X_\lambda^*} \lesssim \sup_{y\in\partial\Omega}\norm{Vu_{in}({\centerdot\,,y})}_{X_\lambda}. \]
    If we take a smooth cut-off function $\eta$ such that $\eta\equiv 1$ in $\supp V$ and $\eta\equiv 0$ in $\partial\Omega$, we have that $V(x)u_{in}(x,y)=V(x)\eta(x)u_{in}(x,y)$ and
    \[ \norm{Vu_{in}({\centerdot\,,y})}_{X_\lambda}\lesssim \norm{\eta u_{in}(\centerdot\,,y)}_{X_\lambda^*}\lesssim 1, \]
    where we have used that multiplication by $V$ is bounded from $X_\lambda^*$ to $X_\lambda$, as showed in section \ref{sect:direct}, and that $u_{in}(\centerdot\,,y)$ is smooth away from $y$ by Theorem 11.1.1 in \cite{hormanderlinear1963}, since it solves $(\Delta+\lambda)u_{in}(\centerdot\,,y)=0$. This proves that $w$ is in $H^1_\text{loc}(\Rd)$\\
     Moreover, since $u_{sc}$ solves the problem \eqref{eq:helm}, it is easy to check that $u = v + w$ solves $\helm u = 0$ in $\mathbb{R}^d\setminus \partial \Omega$ and fullfils the SRC \eqref{eq:src}. Also, since for any $y\in\partial \Omega$, $u_{sc}(\centerdot\,,y)$ solves $\left(\Delta+\lambda\right) u_{sc}(\centerdot\,y)= 0$ in $\Rd\setminus\supp V$, it is smooth in this set by Theorem 11.1.1 in \cite{hormanderlinear1963}, and in particular it is smooth near $\partial \Omega$. Therefore, the limiting values of $w$ on the boundary are just
    \[ \partial_\nu w_{\pm}(x)=\int_{\partial \Omega}\partial_{\nu_x}u_{sc}(x,y)\,f(y)\dd S(y),\quad x\in\partial \Omega, \] and therefore \eqref{eq:bdsingle} is fullfiled. 
\end{proof}
\begin{lemma}\label{lem:double}
    Consider $d\geq3$. Lef $f\in\mathcal{C}(\partial \Omega)$. Then, the double layer potential $u=\mathcal{D}f$  can be extended continuosly to $\partial \Omega$, and we have the limiting values
    \begin{equation}\label{eq:bddouble}
        u_{\pm}(x)=\left(\mathcal{D}f\right)(x)\pm \frac{1}{2}f(x),\quad x\in\partial \Omega,
    \end{equation}
    where the integral exists as an improper integral. Consequently, we have the jump relation $u_+-u_- = f$ on $\partial \Omega$. Furthermore, $u$ a solution in $H^1_{loc}(\Rd)$ to $\helm u = 0$ in $\mathbb{R}^d\setminus \partial \Omega$, it fullfils SRC \eqref{eq:src} and $\partial_\nu u_--\partial_\nu u_+=0$ on $\partial\Omega$.
\end{lemma}
\begin{proof} 
    The proof goes exactly as the proof of lemma \ref{lem:single} above, we just need to make a comment on how to prove the last statement. Indeed, the fact that the double layer potential for the homogeneous Helmholtz equation \[ v(x)=\int_{\partial \Omega}\partial_{\nu_y} u_{in}(x,y)f(y)\dd S(y) \] fulfills that $\partial_\nu v_--\partial_\nu v_+=0$ on $\partial\Omega$ is classical (see for example \cite{coltoninverse2013}). Meanwhile, the function \[ w(x)=\int_{\partial \Omega}\partial_{\nu_y}u_{sc}(x,y)f(y)\dd S(y) \] is smooth away from $\supp V$, since $u_{sc}(\centerdot,y)$ is smooth away from $\supp V$ as well. 
\end{proof}
With lemma \ref{lem:single}, we can find an $H^1(\Omega)$ solution to the problem \eqref{eq:neumann2} for $g\in L^2(\partial \Omega)$ if we can find $\varphi\in L^2(\partial \Omega)$ solving the integral equation $\mathcal{N}\varphi+\frac{1}{2}\varphi=g$. We will solve this equation via Fredholm theory. For simplicity, define the operators
\[ \mathcal{K}=2\mathcal{D},\quad\mathcal{K}^*=2\mathcal{N}, \]
and the equation we are trying to solve can be written as $\left(\mathcal{K}^*+I\right)\varphi = 2g$.
Note that $\mathcal{K}^*$ is the adjoint operator of $\mathcal{K}$ and, by lemma \ref{lem:kernel}, both are compact operators over $L^2(\partial \Omega)$. To prove the existence of a solution, it is enough to prove that the operator $\mathcal{K}^*+I$ is surjective, which by Fredholm alternative is equivalent to proving that $\mathcal{K}+I$ is injective. Indeed, we have the following lemma:
\begin{lemma}
    Suppose that $\lambda>0$ is not a Neumann eigenvalue for the operator $-\Delta+V$ in $\Omega$. Then, 
    \[ \ker(\mathcal{K}+I)=\{0\}. \]
\end{lemma}
\begin{proof}
    Let $\psi\in\ker(\mathcal{K}+I)$, which, by lemma \ref{lem:kernel}, will be continuous on $\partial \Omega$. Now define $v=\mathcal{D}\psi$. By lemma \ref{lem:double}, $v$ is a solution to the problem
    \begin{equation*}
		\begin{cases}
		\left(\Delta+\lambda\right) v = 0 & \textrm{in }\mathbb{R}^d\setminus \overline{\Omega},\\
		v_+ = 0 & \textrm{on }\partial \Omega,\\ 
		v\textrm{ satisfying SRC}.
		\end{cases}
	\end{equation*}
    Then, $v=0$ in $\Rd\setminus \Omega$, by uniqueness of the exterior Dirichlet problem \cite{coltonintegral}. Therefore, $\partial_\nu v_+=0$ on $\partial \Omega$ and, again by lemma \ref{lem:double}, $\partial_\nu v_-=0$. This means that $v$ is as well a solution to the problem
    \begin{equation*}
		\begin{cases}
		\left(\Delta+\lambda-V\right) v = 0 & \textrm{in }\Omega,\\
		\partial_\nu v_- = 0 & \textrm{on }\partial \Omega,\\ 
        \end{cases}
	\end{equation*}
    but, since $\lambda$ is not a Neumann eigenvalue for $-\Delta+V$ in $\Omega$, we have that $v=0$ in $\Omega$, and in particular $v_-=0$. Finally, by lemma \ref{lem:double}, $\psi=u_+-u_-=0$, and the lemma is proved.
\end{proof}
With this last lemma, the proof of Theorem \ref{thm:neu} is concluded.
\bibliographystyle{plain}
\bibliography{fullzoterolibrary}
\end{document}